\documentclass[11pt,letterpaper]{article}
\usepackage{amsfonts, amsmath, amssymb, amscd, amsthm, color, graphicx, mathrsfs, wasysym, setspace, mdwlist, calc, float, mathtools}
\usepackage{tikz}
\usepackage{comment}
\usepackage{setspace}
\usepackage{blindtext}
\usepackage{xcolor}
\usepackage{hyperref}
\usepackage{tocloft}
\usepackage{xcolor}
\usepackage{hyperref}
\hypersetup{linktocpage}
\hypersetup{colorlinks,
    linkcolor={red!50!black},
    citecolor={blue!80!black},
    urlcolor={blue!80!black}}
\usepackage{float}

\newcommand{\NN}{\mathbb{N}}
\newcommand{\ZZ}{\mathbb{Z}}
\newcommand{\RR}{\mathbb{R}}
\newcommand{\HH}{\mathbb{H}}
\newcommand{\CC}{\mathbb{C}}

\renewcommand{\d}{{\rm d}}

\renewcommand{\d}{{\rm d}}

\renewcommand{\phi}{\varphi}
\renewcommand{\l}{\langle}
\renewcommand{\r}{\rangle}

\DeclareMathOperator{\rist}{Rist}
\DeclareMathOperator{\Mon}{Mon}
\DeclareMathOperator{\Aut}{Aut}
\DeclareMathOperator{\Int}{Int}
\DeclareMathOperator{\Fix}{Fix}

\DeclareMathOperator{\supp}{supp}
\DeclareMathOperator{\SL}{SL}
\DeclareMathOperator{\PSL}{PSL}

\DeclareMathOperator{\Sym}{Sym}

\DeclareMathOperator{\Isom}{Isom}

\DeclareMathOperator{\HNN}{HNN}

\DeclareMathOperator{\SO}{SO}

\DeclareMathOperator{\CAT}{CAT}

\newcommand\blfootnote[1]{%
  \begingroup
  \renewcommand\thefootnote{}\footnote{#1}%
  \addtocounter{footnote}{-1}%
  \endgroup
}

\newfont{\eufm}{eufm10}

\newtheorem{thmA}{Theorem}[section]

\newtheorem{corA}[thmA]{Corollary}
\newtheorem{propA}[thmA]{Proposition}

\newtheorem{thm}{Theorem}[section]
\newtheorem{cor}[thm]{Corollary}
\newtheorem{lem}[thm]{Lemma}
\newtheorem{prop}[thm]{Proposition}

\theoremstyle{definition}
\newtheorem{defn}[thm]{Definition}

\theoremstyle{remark}
\newtheorem{rem}[thm]{Remark}
\newtheorem{ex}[thm]{Example}

\title{Boundary dynamics, triple transitivity, and mixed identities in weakly hyperbolic groups}
\author{Ekaterina Rybak}
\date{}
\begin{document}
\maketitle
\vspace{-10mm}
\begin{abstract}
We study the interplay between the algebraic and dynamical properties of groups that admit a general type action on a $\delta$-hyperbolic space such that the induced action on the limit set in the Gromov boundary is faithful. We divide the class of such groups into two subclasses based on a dynamical criterion: groups whose induced action on the limit set is topologically free, and those whose action is not. We prove that satisfying the criterion is equivalent to a purely algebraic property of being mixed identity free, generalizing results of P. Fima, F. Le Ma\^{i}tre, S. Moon, and Y. Stalder from groups acting on trees to groups acting on arbitrary hyperbolic spaces. As one corollary, we obtain that the reduced $C^*$-algebra of such groups is selfless, using the results of N. Ozawa. For the subclass of groups whose action on the limit set is not topologically free, we prove the rigidity result for all $3$-transitive faithful actions and bound the transitivity degree by $3$, generalizing the result of A. Le Boudec and N. Matte Bon. As an application, we show that non-affine simple Kac-Moody groups over finite fields are MIF, answering the question of J. Belk, F. Fournier-Facio, J. Hyde, and M. Zaremsky.
\end{abstract}

\tableofcontents

\section{Introduction}

\subsection{Weakly hyperbolic groups}
One of the generalizations of the class of non-elementary hyperbolic groups is the class of acylindrically hyperbolic groups. Recently, studies of acylindrically hyperbolic groups have attracted a lot of attention. A group is {\it acylindrically hyperbolic} if it admits an acylindrical action of general type on a hyperbolic space. The acylindricity condition is a generalization of properness of the action. We refer the reader to \cite{Osi16} for the definitions. The class of acylindrically hyperbolic groups is large; in particular, it includes all non-elementary hyperbolic and relatively hyperbolic groups, mapping class groups of punctured closed surfaces except for surfaces with genus $g=0$ and fewer than three punctures, $Out(F_n)$ for $n \geq 2$, directly indecomposable right-angled Artin groups, and 1-relator groups with at least 3 generators. Despite being such a broad class, the requirement of the acylindricity of the action is restrictive enough to obtain nontrivial results about the algebraic properties of acylindrically hyperbolic groups. For example, acylindrically hyperbolic groups are closed under taking infinite normal subgroups \cite{Osi16}, are SQ-universal, and have finite amenable radical \cite{DGO}. Finitely generated acylindrically hyperbolic groups have cut points in all of their asymptotic cones \cite{Sis} and have exponential conjugacy growth \cite{HO13}. 

We want to enlarge the class of acylindrically hyperbolic groups even further and study what properties hold when the acylindricity condition is removed. Following the terminology from \cite{MT}, we say
\begin{defn}
A group $G$ is {\it weakly hyperbolic} if it admits a general type action on a geodesic hyperbolic space. An isometric action on a hyperbolic space is {\it of general type} if orbits are unbounded and $G$ contains two loxodromic elements with disjoint fixed points on the Gromov boundary.
\end{defn}

\begin{figure}
  \centering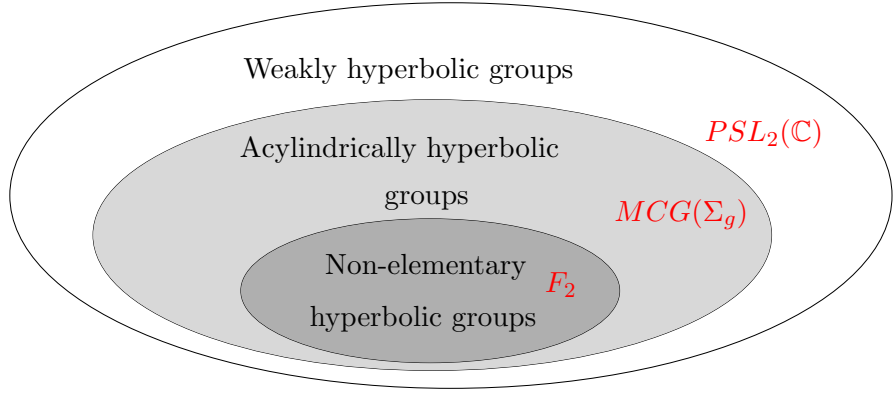\\
   \caption{The class of weakly hyperbolic groups. By $F_2$ we denote the free group on $2$ generators, $MCG(\Sigma_g)$ denotes the mapping class group of any punctured closed surface, except for surfaces with genus
$g=0$ and fewer than three punctures, $PSL_2(\CC)$ denotes the projective special linear group over $\CC$.}\label{classes inclusion}
\end{figure}

We refer the reader to Section \ref{Group actions on hyperbolic spaces} for details on Gromov hyperbolic spaces and the classification of group actions on them.
Examples of weakly hyperbolic groups that are not necessarily acylindrically hyperbolic come from groups acting on trees and, more generally, groups of isometries of hyperbolic spaces. 

Many properties of particular examples of weakly hyperbolic groups are well documented in the literature. Various results on random walks in weakly hyperbolic groups were obtained in \cite{MT}, \cite{Sun}, \cite{BMSS}, \cite{AMS}, \cite{AS}, \cite{GTT22}, \cite{Cho}, and \cite{Hor}. However, not many group structural results are known about the whole class. Using a simple ping-pong argument, one can prove that any weakly hyperbolic group contains a nonabelian free subgroup (see \cite[8.2.F]{Gro87} or \cite[Thm. 2.7]{Ham}). In \cite{Osi17}, it is proved that for every invariant random subgroup $H$ of a countable weakly hyperbolic group $G$, when $G$ acts on some hyperbolic space via a general type action, the induced action of $H$ either has the same limit set as $G$ or all orbits of $H$ are bounded. In particular, every normal subgroup of a weakly hyperbolic group satisfies this dichotomy. In \cite[Thm 1.2]{GOR} it is proved that the Frattini subgroup of a countable weakly hyperbolic group acts on a hyperbolic space with bounded orbits, and thus has infinite index.

 A natural question to ask is what other interesting algebraic properties are satisfied by every weakly hyperbolic group. Finding such properties is challenging, as having just the general type action without any properness condition is not very limiting. Therefore, we want to impose an additional condition on the group action.

Recall that the limit set $\Lambda_S(G)$ in the Gromov boundary of a hyperbolic space $S$ is the closure of the set of accumulation points of orbits under the $G$-action. 

\begin{defn}
We say that a group $G$ is {\it faithful weakly hyperbolic} if it admits a general type action on a hyperbolic space $S$ such that the induced $G$-action on $\Lambda_S(G)$ is faithful. 
\end{defn}

Note that not every weakly hyperbolic group is faithful weakly hyperbolic. A source of examples of weakly hyperbolic groups that are not faithful weakly hyperbolic groups is given in Proposition \ref{WH not FWH groups}. However, by Proposition \ref{faithful weakly hyperbolic quotient}, any weakly hyperbolic group has a faithful weakly hyperbolic quotient. Here, $E_S(G)$ is the {\it elliptic radical} with respect to the action of $G$ on $S$. It is the set of all elements of $G$ that act on $\Lambda_S(G)$ trivially. We refer the reader to Section \ref{elliptic radical} for the basic properties of $E_S(G)$.

\begin{propA}\label{faithful weakly hyperbolic quotient}
Suppose that a group $G$ admits a general type action on a hyperbolic space $S$. Then $G/E_S(G)$ is faithful weakly hyperbolic.
\end{propA}

Let $X$ be a topological space. Recall that an action of a group $G$ on $X$ by homeomorphisms is {\it topologically free} if the set of fixed points $\Fix_X(g)$ of every non-trivial element $g \in G$ has empty interior.

\begin{defn}
A weakly hyperbolic group $G$ is called {\it lim-free} if it admits a general type action on a hyperbolic space $S$ such that the induced action on $\Lambda_S(G)$ is topologically free.
\end{defn}

We have an obvious inclusion: Lim-free weakly hyperbolic groups are faithful weakly hyperbolic groups. While the property of being weakly hyperbolic is not closed under taking non-trivial normal subgroups, both properties of being faithful weakly hyperbolic and of being lim-free are closed under taking non-trivial normal subgroups by Lemma \ref{lim-free passes to normal subgroups}. We will see that lim-free weakly hyperbolic groups and faithful weakly hyperbolic groups that are not lim-free are two distinct classes. Groups from the former class include acylindrically hyperbolic groups without non-trivial finite normal subgroups (see Corollary \ref{acylindrically hyperbolic groups are lim-free}) and behave similarly to them, while groups from the latter class have very different algebraic properties. 

\subsection{MIF weakly hyperbolic groups}

Let $F_n$ denote a free group of rank $n$. Recall that a group $G$ satisfies a {\it mixed identity} $w=1$ for some $w \in G * F_n$ if every homomorphism $G * F_n \rightarrow G$ that is the identity on $G$ sends $w$ to $1$. A mixed identity $w=1$ is {\it non-trivial} if $w$ is a non-trivial element of $G * F_n$, and a group $G$ is {\it mixed identity free} (or MIF for brevity) if it does not satisfy any non-trivial mixed identity. 
The property of being MIF imposes strong restrictions on the algebraic structure of $G$. If $G$ is MIF, then $G$ has infinite conjugacy classes, it is directly indecomposable, and it has infinite girth (see \cite[Prop. 5.4]{OH}). A countable group $G$ is MIF if and only if $G$ and $G * F_n$ are universally equivalent as $G$-groups for all $n \in \NN$ (see \cite[Prop. 5.3]{OH}). A surprising connection between MIF groups and the Boone-Higman conjecture was recently discovered in \cite{BFFHZ}.

In our first main theorem, part (b) reveals a rigidity phenomenon: properties of one general type action impose strong constraints on all general type actions. Furthermore, part (c) shows that the geometrically defined class of lim-free groups can be characterized within the class of faithful weakly hyperbolic groups by a purely algebraic condition.

\begin{thmA}\label{main_theorem}
For any faithful weakly hyperbolic  group $G$, the following conditions are equivalent:
\begin{enumerate}
\item[(a)] $G$ is lim-free;
\item[(b)] for any general type action of $G$ on a hyperbolic space with trivial elliptic radical, the induced action on the limit set is topologically free;
\item[(c)] $G$ is MIF.
\end{enumerate}
\end{thmA}

\begin{figure}
  \centering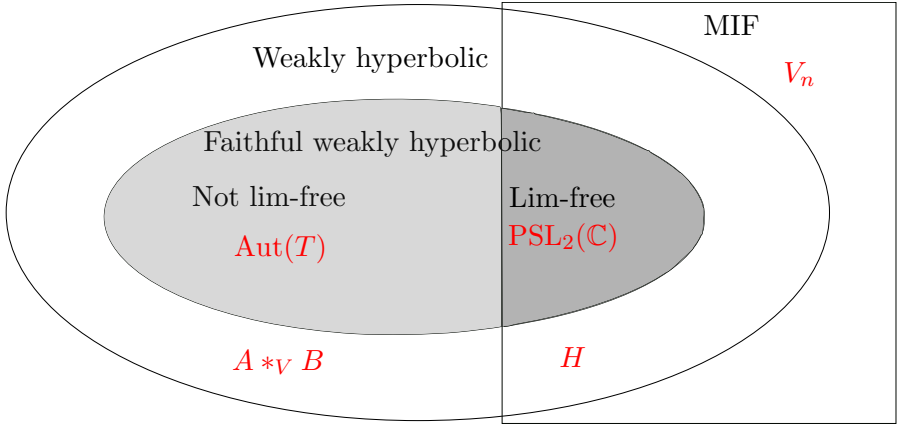\\
   \caption{The interaction between classes of weakly hyperbolic groups and MIF groups. Here $A*_VB$ denotes the verbal product of non-trivial groups $A$ and $B$ with respect to a nonempty set of words $V$, where at least one of the factors is weakly hyperbolic (see Example \ref{Verbal products}), $V_n$ denotes the  Higman-Thompson group (see \cite[Thm. 3.5]{BEH}), and $H$ stands for a group constructed in Example \ref{example H}.}\label{MIF and not MIF}
\end{figure}

Considering mixed identities in weakly hyperbolic groups was inspired by the results about groups acting on simplicial trees. An action of a group $G$ on a tree $T$ is called {\it minimal} if there are no nontrivial invariant subtrees. Combining \cite[Proposition 3.7]{BM} with their results, P. Fima, F. Le Ma\^{i}tre, S. Moon, and Y. Stalder prove that if the action of a group $G$ on a simplicial tree $T$ is faithful, minimal, and of general type, then $G$ is MIF if and only if the induced action on $\partial T$ is topologically free \cite[Theorem B]{FMMS}. Their result can be obtained as a corollary of Theorem \ref{main_theorem} (See Corollary \ref{previous results about MIF in trees}).

One of the applications of Theorem \ref{main_theorem}, not covered by \cite[Theorem B]{FMMS}, is the proof that infinite simple Kac-Moody groups over finite fields are MIF (see Example \ref{ex BM}). This partially answers Question 3.6 from \cite{BFFHZ}. 

Let $G$ be a group and $w(t) \in G*\langle t\rangle$ be a mixed identity. The solution set $\{g \in G \mid w(g)=1\}$ is called { \it the primitive solution set corresponding to $w(t)$}. {\it The Zariski topology} on $G$ is defined by taking the collection of primitive solution sets to be a subbasis for the closed sets of the topology. A Zariski-closed subgroup (or more generally, a subset) of $G$ is called {\it algebraic}. The natural question to ask is what proper subgroups of a group $G$ can be algebraic. If $G$ is torsion-free non-elementary hyperbolic, then all such subgroups are cyclic \cite{KM}. In \cite{Jac}, structural conditions for non-elementary algebraic subgroups of acylindrically hyperbolic groups were obtained. By a similar argument as in the proof of Theorem \ref{main_theorem}, we get

\begin{corA}\label{algebraic subgroups}
Let $G$ be a lim-free weakly hyperbolic group. Then $G$ does not contain any non-trivial proper normal Zariski-closed subgroups.
\end{corA}

Non-lim-free faithful weakly hyperbolic groups also have interesting common purely algebraic properties. For example, they contain infinite direct sums of infinite non-abelian subgroups (see Corollary \ref{infinite direct products}). Recall that the {\it monolith} $\Mon(G)$ of a group $G$ is the intersection of all non-trivial normal subgroups of $G$, and $G$ is said to be {\it monolithic} if $\Mon(G) \neq \{1\}$. As a corollary of \cite[Corollary 1.7]{GR}, we obtain the following. 

\begin{propA}\label{monolithic groups}
Let $G$ be a faithful weakly hyperbolic group. If $G$ is not lim-free, then it is monolithic with an infinite simple monolith.
\end{propA} 
In this way, non-lim-free faithful weakly hyperbolic groups are very different from acylindrically hyperbolic groups, whose monolith is always trivial (see Proposition \ref{monolith is trivial}). This allows us to obtain the result of M. Hull and D. Osin \cite[Corollary 5.10]{OH} that acylindrically hyperbolic groups without nontrivial finite normal subgroups are MIF as a corollary of Theorem \ref{main_theorem} and Proposition \ref{monolithic groups} (see Corollary \ref{acylindrically hyperbolic groups are lim-free}). 

Another corollary of Proposition \ref{monolithic groups} gives us a large class of lim-free weakly hyperbolic groups.

\begin{corA}\label{res finite groups}
Every residually finite faithful weakly hyperbolic group is lim-free. In particular, every finitely generated linear faithful weakly hyperbolic group is lim-free.
\end{corA}

The famous Tits Alternative states that a finitely generated linear group is either virtually solvable or contains a nonabelian free subgroup. As another corollary of Theorem \ref{main_theorem}, we obtain an analogous statement for all subgroups of $\PSL_2(\CC)$.

\begin{corA}\label{generalized Tits Alternative}
Every subgroup of $\PSL_2(\CC)$ is either virtually solvable or MIF.
\end{corA}

We should note that in \cite{GGS}, it is proved that any center-free unbounded and non-virtually solvable countable subgroup of $\SL_2(k)$ is highly transitive, where $k$ is a local field, which implies being MIF by the result of \cite{OH}. In \cite{Tom}, Tomanov gives necessary and sufficient conditions for Zariski-dense subgroups of linear groups to be MIF. Corollary \ref{generalized Tits Alternative} does not require any assumptions on subgroups of  $\PSL_2(\CC)$ and is not an immediate consequence of Corollary \ref{res finite groups}, \cite{Tom}, or \cite{GGS}.

A property called {\it selflessness} was recently introduced for a $\mathrm{C}^*$-algebra by L. Robert. It implies many important regularity properties such as simplicity, stable rank one in the tracial setting, and strict comparison. 
We refer the reader to $\cite{Rob}$ for the definitions. In \cite{AGKEP}, Tattwamasi Amrutam, David Gao, Srivatsav Kunnawalkam Elayavalli, and Gregory Patchell proved that acylindrically hyperbolic groups without finite normal subgroups have selfless reduced $C^*$-algebras. In their proof, they used a stronger quantitative version of the property of being MIF and the rapid decay property. This attracted attention to the quantitative version of being MIF, which resulted in a series of works: \cite{BS}, \cite{AG}, \cite{Vid}. Later, in his remarkable paper, Ozawa obtained a dynamical condition for the group to have a selfless reduced $C^*$-algebra. Using the results of $\cite{Oza}$, we obtain 

\begin{corA}\label{selfless groups}
If a weakly hyperbolic group $G$ is lim-free, its reduced group $C^*$-algebra $C_\lambda^*(G)$ is selfless.
\end{corA}

Corollary \ref{selfless groups} gives us the first examples of finitely presented simple groups with selfless reduced $C^*$-algebras.

\begin{propA}\label{Kac-Moody}
Non-affine simple finitely presented Kac-Moody groups over finite fields constructed in \cite{CR} are MIF, and their reduced $C^*$-algebras are selfless.
\end{propA}

\subsection{Rigidity of $3$-transitive actions}

Recall that a $G$-action on some set $\Omega$ is {\it $k$-transitive} for some $k \in \NN$ if for any two $k$-tuples of distinct points $(\omega_1, \dots, \omega_k) \in \Omega^k$ and  $(v_1, \dots, v_k) \in \Omega^k$, there exists $g \in G$ such that $g\omega_i=v_i$ for all $i=1, \dots, k$.  As discussed earlier, Theorem \ref{main_theorem} gives a rigidity phenomenon for general type actions of lim-free weakly hyperbolic groups. If a group is faithful weakly hyperbolic but not lim-free, we obtain a rigidity phenomenon for all $3$-transitive actions on any set, as our second main theorem shows.  

\begin{defn}
Two actions of a group $G$ on sets $\Omega$ and $\Omega^{\prime}$ are {\it conjugate} if there exists a bijective $G$-equivariant map $\Omega \rightarrow \Omega^{\prime}$.
\end{defn}

\begin{thmA}\label{thm about conjugate orbits}
Suppose that the action of a group $G$ on a hyperbolic space $S$ is of general type and that the induced action of $G$ on $\Lambda_S(G)$ is faithful and not topologically free. 
Assume that $G$ acts faithfully and 3-transitively on a set $\Omega$. Then there exists a $G$-orbit $\mathcal{O} \subset \Lambda_S(G)$ such that the actions of $G$ on $\Omega$ and on $\mathcal{O}$ are conjugate.
\end{thmA}

This theorem is the analogue of \cite[Theorem 1.4]{BM} for group actions on simplicial trees. We generalize their arguments to groups acting on arbitrary hyperbolic spaces and obtain their result as a corollary (see Corollary \ref{conjugate orbits for trees}).  

The {\it transitivity degree} of a group $G$ is the maximum $k$ such that $G$ admits a $k$-transitive faithful action on some set. Among infinite groups, examples of groups with transitivity degree $1$, $2$, $3$, and $\infty$ are known. Whether there exists an infinite group with transitivity degree $4$ is an open question. The next corollary shows that if such an example exists, it cannot be a faithful weakly hyperbolic group that is not lim-free.     

\begin{corA}\label{td is leq 3}
Let $G$ be a faithful weakly hyperbolic group that is not lim-free. Then the transitivity degree of $G$ is at most $3$.
\end{corA}

Recall that an action of a group $G$ on a topological space $X$ is {\it topologically transitive} if, for every pair of nonempty open subsets $U$ and $V$ of $X$, there is an element $g \in G$ such that $g(U) \cap V \neq \emptyset$. It is {\it topologically $k$-transitive} for $k \in \NN$ if the induced action of $G$ on the $k$-fold Cartesian product $X^k$ is topologically transitive.

Since Theorem \ref{thm about conjugate orbits} tells us that the limit set ``sees" all $3$-transitive actions of a faithful weakly hyperbolic group $G$ that is not lim-free, it suffices to examine just the $G$-action on the limit set to determine whether the transitivity degree of $G$ is $3$ or less, as the following corollary shows.

\begin{corA}\label{topological 3-transitivity}
Let $G$ be a faithful weakly hyperbolic group that is not lim-free. If the transitivity degree of $G$ is $3$, then for any general type action of $G$ on any hyperbolic space $S$, if the induced action on $\Lambda_S(G)$ is faithful, it is topologically $3$-transitive.
\end{corA}

The paper is organized as follows. In Section \ref{preliminary}, we provide the necessary background and preliminary results and prove Proposition \ref{faithful weakly hyperbolic quotient} in Subsection \ref{elliptic radical}. In Subsection \ref{boundary dynamics and mixed identities}, we prove Theorem \ref{main_theorem}, Corollary \ref{algebraic subgroups}, Proposition \ref{monolithic groups}, Corollary \ref{res finite groups}, Corollary \ref{generalized Tits Alternative}, and Corollary \ref{selfless groups}. In Subsection \ref{conjugate orbits}, we prove Theorem \ref{thm about conjugate orbits}, Corollary \ref{td is leq 3}, and Corollary \ref{topological 3-transitivity}. In Subsection \ref{examples}, we prove Proposition \ref{Kac-Moody} and discuss some examples and the necessity of the conditions of Theorem \ref{main_theorem} and Theorem \ref{thm about conjugate orbits}.

\paragraph{Acknowledgments.}The author is grateful to Denis Osin and Spencer Dowdall for providing guidance throughout the course of this project.
The author is thankful to Francesco Fournier-Facio for comments on the first version, for drawing her attention to Kac–Moody groups, and for helpful discussions concerning Example \ref{ex BM}.

\section{Preliminary results}\label{preliminary}

\paragraph{Notation.}
Throughout the paper, we denote by $S$ a geodesic $\delta$-hyperbolic metric space, and by $\d_S$ the metric on $S$. All group actions on $S$ are by isometries. For a group $G$ acting on $S$ and any $s\in S$, let $Gs=\{ gs\mid g\in G\}$ denote the $G$-orbit of the point $s$.

For group elements $g$ and $h$ we denote by $g^h$ the conjugate $h^{-1}gh$.
If a group $G$ acts on a set $X$, we denote by $\Fix_X(G)$ the set of fixed points of $G$ in $X$, and by $G_x$ the stabilizer subgroup of $x \in X$ under the action of $G$. If $U$ is a subset of a topological space $X$, we denote by $\Int(U)$ the interior of $U$ in $X$.

\subsection{Group actions on hyperbolic spaces}\label{Group actions on hyperbolic spaces}
We refer the reader to \cite[Chapter 3]{DSU} and \cite{BS07} for more details on the material discussed in this subsection.

\begin{defn}\label{Gromov product}
Let $S$ be a geodesic metric space. Recall that the \textit{Gromov product} of two points $x, y \in S$ with respect to a point $o \in S$ is defined by
$$
\l x|y\r_o:=\frac{1}{2}(\d_S(x, o)+\d_S(y, o)-\d_S(x, y)).
$$
\end{defn}
\begin{defn}[Gromov hyperbolicity]\label{four point condition}
A geodesic metric space $S$ is  {\it $\delta$-hyperbolic} if for any points $x, y, z, o \in S$, we have
$$
\l x|y \r_o \geq \min\left\{\l x|z \r_o, \l z|y\r_o\right\} - \delta.
$$ 
\end{defn}
If $Y \subseteq S$ is a subspace, then for any constant $K \geq 0$, we denote the closed $K$-neighborhood of $Y$ in $S$ by $\mathcal{N}_K(Y)=\left\{x \in S \mid \d_S(x, Y) \leq K\right\}$.

\begin{defn}[Rips hyperbolicity]\label{rips def} 
Fix $\delta \geq 0$. A geodesic metric space $S$ is $\delta$-hyperbolic if given any $x, y, z \in S$ and any geodesics $\alpha, \beta, \gamma$ between them, we have $\gamma \subseteq \mathcal{N}_\delta(\alpha \cup \beta)$.
\end{defn}

A geodesic metric space is Rips hyperbolic if and only if it is Gromov hyperbolic (see, for example, \cite[Lemma 1.46]{KS}). In this work, we choose $\delta$ in a way that $S$ satisfies both Definitions \ref{four point condition} and Definition \ref{rips def} simultaneously.

A sequence $\left(x_i\right)$ of elements of $S$ \textit{converges at infinity} if $\l x_i|x_j \r_o \rightarrow \infty$ as $i, j \rightarrow \infty$ for some, or equivalently for any, point $o\in S$. Two such sequences $\left(x_i\right)$ and $\left(y_i\right)$ are \textit{equivalent} if $\l x_i|y_j \r_o \rightarrow \infty$ as $i, j \rightarrow \infty$. The {\it Gromov boundary} of $S$, denoted by $\partial S$, is defined as the set of equivalence classes of sequences converging at infinity. If $a$ is the equivalence class of $\left(x_i\right)$, we say that the sequence $(x_i)$ converges to $a$. There is a natural topology on $\widehat{S}=S \cup \partial S$ extending the topology on $S$ such that $S$ is dense in $\widehat{S}$, which will be discussed later.  The definition of $\partial S$ is independent of the choice of the basepoint, and any isometric group action on $S$ extends to a continuous action on $\widehat S$. The Gromov product can be extended to $\widehat{S}$ in the following way.
\begin{defn}\label{Gromov product of boundary points}
Let $a, b \in \partial S$ and $o, x \in S$. 
\begin{enumerate}
\item The {\it Gromov product} of  $a$ and $b$ with respect to $o$ is defined by
\begin{equation*}
\l a| b\r_o = \inf \left\{ \liminf_{i, j \rightarrow +\infty} \l x_i|y_j\r_o \; \bigm | \; (x_i) \rightarrow a, (y_j) \rightarrow b \right \}.
\end{equation*}
\item The {\it Gromov product} of $a$ and $x$ with respect to $o$ is defined by
\begin{equation*}
\l a|x\r_o = \inf \left \{ \liminf_{i \rightarrow \infty} \l x_i|x\r_o \; \bigm | \; (x_i) \rightarrow a \right \}.
\end{equation*}
\end{enumerate}
\end{defn}
The next lemma shows that taking the $\liminf$ in the definition of the Gromov product of the boundary points is the same as taking $\limsup$ up to a constant and that the boundary points also satisfy the four-point condition.

\begin{lem}[\cite{BS07}, Lemma 2.2.2]\label{lemma 2.2.2}
Let $S$ be a $\delta$-hyperbolic space. Let $o\in S$, and let $a, b, c \in \partial S$.
\begin{enumerate}
\item For arbitrary sequences $\left(x_i\right) \rightarrow a,\left(y_i\right) \rightarrow b$, we have
$$
\l a| b\r_o \leq \liminf _{i \rightarrow \infty}\l x_i| y_i\r_o \leq \limsup _{i \rightarrow \infty}\l x_i|y_i\r_o \leq \l a| b\r_o+2 \delta .
$$
\item 
$$
\l a|b \r_o \geq \min\left\{\l a|c \r_o, \l c|b\r_o\right\} - \delta.
$$ 

\end{enumerate}
\end{lem}

The Gromov product of distinct points of the boundary is always finite, i.e., if $a, b \in \partial S$, then $\l a|b\rangle_o = \infty$ if and only if $a=b$. It is also preserved by isometries of $S$, i.e., if $g$ is an isometry of $S$, $a,b \in \widehat{S}$ then $\l ga|gb \r_{go} = \l a|b\r_o$ (see \cite[Section 3.4.1]{DSU}).

For any point $a \in \partial S$, $o \in S$, and $r \in \mathbb R$, denote by
$$\mathcal U_o(a,r) = \left \{ s \in \widehat S \mid \l a|s\r_o > r \right \}.$$

The subset $\mathcal U_o(a,r)$ is open in the topology on $\widehat{S}$, and $\{ \mathcal U_o(a,r) \mid r \geq 0 \}$ form a neighborhood base for the topology on $\widehat{S}$ at $a$ (see \cite[Remark 3.4.15 (II)]{DSU}).

The \textit{limit set} $\Lambda_S(G)$ of $G$ is defined as the intersection of the closure of $Gs$ in $\widehat S$ with $\partial S$. This definition is independent of the choice of $s\in S$.

An element $g \in G$ is said to be {\it loxodromic} if the group $\l g\rangle $ has unbounded orbits and fixes exactly two points on $\partial S$ (equivalently, $|\Lambda_S(\l g\r)|=2$). By $\mathcal{L}_S(G)$ we denote the set of all loxodromic elements of $G$. 

The following lemma establishes the standard notation for fixed points of loxodromic elements, motivated by the so-called north-south dynamics (see \cite[Lemma 8.1.G]{Gro87} or \cite[Theorem 6.1.10]{DSU}).

\begin{lem} \label{north-south}
The fixed points of $g \in \mathcal{L}_S(G)$ can be denoted by $g^+$ and $g^-$ so that
$$
\lim _{n \rightarrow \infty} g^n s=g^{+} \;\;\;  \forall s \in \widehat S \setminus \left\{g^{-}\right\}\;\;\;\;\; {\rm and }\;\;\;\;\; \lim _{n \rightarrow \infty} g^{-n} s=g^{-} \;\;\;  \forall s \in \widehat S \setminus \left\{g^{+}\right\}
$$
where the convergence in the first (respectively, second) limit is uniform on sets whose closure does not contain $g^{-}$ (respectively, $g^+$).
\end{lem}

\begin{rem} \label{congugation}
 If $g \in \mathcal{L}_S(G)$, then $fgf^{-1} \in \mathcal{L}_S(G)$ for any $f\in G$. Moreover, we have $(fgf^{-1})^{+}=fg^+$ and  $(fgf^{-1})^{-}=fg^-$.
\end{rem}

We denote the set of endpoints of loxodromic elements by $ 
\mathcal{H}_S(G)=\left\{ g^\pm \mid g \in \mathcal{L}_S(G) \right\}$.

Possible actions of groups on hyperbolic spaces can be classified as follows according to the cardinality of $\Lambda_S(G)$ (see \cite[Sections 8.1-8.2]{Gro87}, \cite{Ham}, or \cite[Chapter 6]{DSU}):
\begin{enumerate}
\item \label{elliptic action} ({\it elliptic action}) $|\Lambda_S(G)|=0$. Equivalently, $G$ has bounded orbits.
\item \label{parabolic action} ({\it parabolic action}) $|\Lambda_S(G)|=1$. Equivalently, $G$ has unbounded orbits, and $\mathcal{L}_S(G)=\emptyset$. In this case $\Fix_{\partial S}(G)=\Lambda_S(G)$.
\item \label{lineal action} ({\it lineal action}) $|\Lambda_S(G)|=2$. Equivalently, $\mathcal{L}_S(G) \neq \emptyset$, and any two loxodromic elements have the same fixed points on $\partial S$. In this case $\Fix_{\partial S}(G) \subseteq \Lambda_S(G)$.
\item \label{non-elementary} $|\Lambda_S(G)|=\infty$. Then $\mathcal{L}_S(G) \neq \emptyset$. In turn, this case breaks into two subcases.
	\begin{enumerate}
	\item \label{quasi-parabolic action} ({\it quasi-parabolic action}) $\Fix_{\partial S}(G) \neq \emptyset$. Then $G$ fixes a unique point of $\partial S$.
	\item \label{general type action} ({\it general type action}) $\Fix_{\partial S}(G)=\emptyset$. Equivalently, $G$ contains at least two (or infinitely many) loxodromic elements with disjoint fixed points.
	\end{enumerate}
\end{enumerate}

The action of $G$ is called {\it elementary} in cases \ref{elliptic action} -- \ref{lineal action} and {\it non-elementary} in case \ref{non-elementary}.

Recall that an action of a group $G$ by homeomorphisms on a topological space $X$ is {\it minimal} if there is no proper closed nonempty $G$-invariant subset $Y \subset X$.

\begin{cor}\cite[Corollary 7.4.1]{DSU}\label{action on the limit set is minimal}
Suppose that the action of a group $G$ on a hyperbolic space $S$ is of general type. Then $\Lambda_S(G)$ is the smallest nonempty closed $G$-invariant subset of $\partial S$. Therefore, the action of $G$ on $\Lambda_S(G)$ is minimal.
\end{cor}

The following can be found in \cite[Section 8.2]{Gro87}; alternatively, see \cite[Theorems 2.6, 2.9]{Ham} or \cite[Propositions 7.4.6, 7.4.7]{DSU}.

\begin{lem}\label{set of fixed points of loxodromic elements is dense}
Suppose that a group G acts non-elementarily on a hyperbolic space $S$. Then $\mathcal{H}_S(G)$ is dense in $\Lambda_S(G)$. Moreover, if the action is of general type, then $\{(g^+, g^-) \mid g \in \mathcal{L}_S(G)\}$ is dense in $\Lambda_S(G) \times \Lambda_S(G)$.
\end{lem}

The following lemma follows from \cite[Theorem 4.5]{Osi17}, where a much more general fact is proved.

\begin{lem}\label{dichotomy}
Suppose that a group $G$ admits a general type action on a hyperbolic space $S$. Then the induced action of any normal subgroup $N \leq G$ either has bounded orbits in $S$ or is also of general type. If the action of $N$ on $\Lambda_S(G)$ is nontrivial, then $\Lambda_S(N)=\Lambda_S(G)$.
\end{lem}

\begin{rem}
The proof of Lemma \ref{dichotomy} remains true for an arbitrary group $G$ with a general type action on a hyperbolic space $S$, not necessarily countable as stated in \cite[Theorem 4.5]{Osi17}.
\end{rem}

For any $\lambda \geq 1$ and $c\in \RR$, a map of metric spaces $\phi:\left(X, \d_X\right) \rightarrow\left(Y, \d_Y\right)$ is a {\it $(\lambda, c)$-quasi-isometric embedding} if for all $x, y \in X$

$$
\frac{1}{\lambda} \d_X(x, y)-c \leq \d_Y(\phi(x), \phi(y)) \leq \lambda \d_X(x, y)+c.
$$

A {\it $(\lambda, c)$-quasi-geodesic} is a $(\lambda, c)$-quasi-isometric embedding $\phi$ of an interval $I \subseteq \RR$ into $S$. In this work, we will not distinguish between quasi-geodesics and their images when the parametrization is not important. If the constants $\lambda$ and $c$ of $\phi$ are not important, we write that $\phi$ is {\it a quasi-geodesic}.

Let $Y \subseteq S$, and let $\sigma \geq 0$. We say that $Y$ is {\it $\sigma$-quasiconvex in $S$} if every geodesic with endpoints in $Y$ is contained in the $\sigma$-neighborhood of $Y$. A subset $Y \subset S$ is said to be {\it quasiconvex} if it is $\sigma$-quasiconvex for some $\sigma \geq 0$. Any ($\lambda, c$)-quasi-geodesic in $S$ is quasiconvex by the Morse Lemma. 

\begin{lem} [Morse Lemma]\label{Morse lemma}
Let $S$ be a $\delta$-hyperbolic metric space, and fix $\lambda \geq 1$ and $c \geq 0$. There exists a constant $\sigma$ depending only on $\delta, \lambda$, and $c$ such that if $\gamma_1$ and $\gamma_2$ are ($\lambda, c$)-quasi-geodesics in $S$ with the same endpoints, then $\gamma_1 \subseteq \mathcal{N}_\sigma\left(\gamma_2\right)$.
\end{lem}

Let $Y \subset S$, and let $x \in S$. The {\it nearest-point projection} of $x$ onto $Y$ is a point $\overline{x}$ such that $\d_S(x, \overline{x}) \leq \d_S(x, Y)+1$. When $Y$ is $\sigma$-quasiconvex, the nearest-point projection is coarsely well-defined: for
different choices $\overline{x}_1$, $\overline{x}_2$ of points in $Y$ nearest to $x$ we have $\d_S(\overline{x}_1, \overline{x}_2) \leq \max(2, 2\sigma + 9\delta)$ (see \cite[Remark 1.103]{KS}).

Let $x$ and $y$ have nearest-point projections $\overline{x}$ and $\overline{y}$ onto a quasi-geodesic segment $\gamma$. The following lemma shows that if $\overline{x}$ and $\overline{y}$ are $O(\delta)$ apart, then, up to additive error, the geodesic from $x$ to $y$ goes from $x$ to $\overline{x}$, then runs along the geodesic from $\overline{x}$ to $\overline{y}$, and then heads back out to $y$.

\begin{figure}
\centering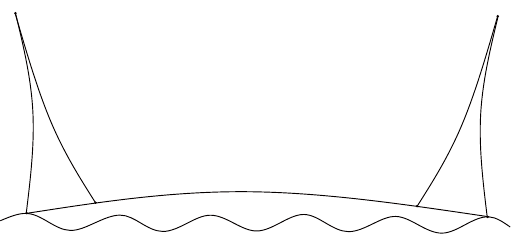\\
  \caption{For Lemma \ref{projection distanse formula}}\label{sgp}
\end{figure}

\begin{lem}\label{projection distanse formula}
Let $x, y \in S$ and let $\gamma \subset S$ be a $(\lambda, c)$-quasi-geodesic segment. Let $\overline{x}$ be a nearest-point projection of $x$ on $\gamma$ and let $\overline{y}$ be a nearest-point projection of $y$ on $\gamma$. Then there exists a constant $M=M(\lambda, c, \delta)$ such that if $\d_S(\overline{x}, \overline{y})>14\delta+2M$ then 
$$
\d_S(x, y) \geq \d_S(x, \overline{x})+\d_S(\overline{x}, \overline{y})+\d_S(\overline{y}, y)-24 \delta-4M.
$$
\end{lem}
\begin{proof}
Consider the geodesic segment $[\overline{x}, \overline{y}]$, and let $p_x$ and $p_y$ be nearest-point projections to $[\overline{x}, \overline{y}]$ of $x$ and $y$, respectively.   
By \cite[Proposition 3.6]{Mah}, if $\gamma$ lies in a bounded neighborhood of a geodesic, then the nearest-point projections of any point to $\gamma$ and to the geodesic are a bounded distance apart. Combining this with the Morse Lemma, we obtain that $p_x$, $p_y$ lie within  $M=M(\lambda, c, \delta)$ of $\overline{x}$ and $\overline{y}$. Since $\d_S(\overline{x}, \overline{y})>14\delta+2M$, we have $\d_S(p_x, p_y)>14\delta$. Therefore, by \cite[Proposition 3.4]{Mah}, the following holds
$$
\d_S(x, y) \geq \d_S(x, p_x)+\d_S(p_x, p_y)+\d_S(p_y, y)-24 \delta.
$$
Using that $\d_S(\overline{x}, p_x)\leq M$ and $\d_S(\overline{y}, p_y)\leq M$, we obtain the conclusion of the lemma.
\end{proof}

\begin{defn}
Let $g \in G$. The {\it asymptotic translation length} of $g$ in $S$ is
$$
\tau_S(g):=\lim _{n \rightarrow \infty} \frac{\d_S\left(s_0, g^n s_0\right)}{n}
$$
for some (equivalently, any) $s_0 \in S$.
\end{defn}

An element $g \in \mathcal{L}_S(G)$ if and only if $\tau_S(g)>0$ \cite[ Chapitre 10, Proposition 6.3]{french book}, and it acts on $S$ by translation along a quasi-geodesic axis which connects the two limit points $g^+$ and $g^-$ of $g$ in $\partial S$.

If $\tau_S(g)$ is sufficiently large, such an axis can be chosen uniformly, as shown in the following lemma, which summarizes results from \cite[Section 3]{Cou16}.

\begin{lem}\label{uniform axis}
Let $G$ act on a $\delta$-hyperbolic space $S$. Suppose $g \in \mathcal{L}_S(G)$ and $\tau_S(g) \geq L_S \delta-16 \delta$, where $L_S$ depends only on $\delta$ (and is more explicitly described in \cite[Definition 2.8]{Cou16}). Then there exists a $(2, \delta)$-quasi-geodesic $L_g$ in $S$ which connects the limit points $g^+$ and $g^-$ of $g$ in $\partial S$. $L_g$ is preserved by $g$ and is $(\delta+8 \delta)$-quasiconvex. The quasi-geodesic $L_g$ is rectifiable, i.e., the length $l(\gamma)$ is well-defined for any segment $\gamma \subset L_g$.
\end{lem}

We will call such a uniform quasi-geodesic axis a {\it standard quasi-geodesic axis} of $g$. Note that we can always increase the asymptotic translation length by taking the appropriate power of $g$. Thus, the standard quasi-geodesic axis exists for every loxodromic element in $G$, up to passing to powers.

The following lemma is a reformulation of the “local-to-global'' property in terms of Gromov products and the reverse triangle inequality for the orbits of loxodromic elements. 

\begin{lem}\label{loxodromic criterion}
For every $A \in \NN$ and every isometry $g \in G$ of $S$, if 
$$
\langle gs \mid g^{-1}s \rangle_s \leq A \text{ and } \d_S(s, gs) > 4A+24\delta,
$$
for some $s \in S$, then $g \in \mathcal{L}_S(G)$.
\end{lem}
\begin{proof}
    Consider the sequence $s, gs, g^2s, \dots, g^ns$ for some $n \in \NN$. By \cite[Lemma 2.1]{Spencer}, we have 
    \begin{equation}\label{sum of dist}
    \d_S(s, g^ns) \geq \frac{1}{2}\sum_{j=1}^n\d_S(g^{j-1}s, g^js).
    \end{equation}
    Letting $n \rightarrow \infty$ in (\ref{sum of dist}), and using that $\d_S(s, gs) > 4A+24\delta$, we conclude that 
    $\tau_S(g)\geq 2A+12\delta$. Since $\tau_S(g) > 0$, $g \in \mathcal{L}_S(G)$.
\end{proof}

\begin{prop}\label{elements of parabolic subgroups don't move some points far}
Suppose that the action of a group $G$ on a hyperbolic space $S$ is parabolic. Then for every $g_1, g_2 \in G$, there exists $s \in S$ such that $\d_S(g_1s, s) \leq 56\delta$ and $\d_S(g_2s,s) \leq 56\delta$.
\end{prop}

For the proof, we will need the following technical lemma. 

\begin{figure}
\centering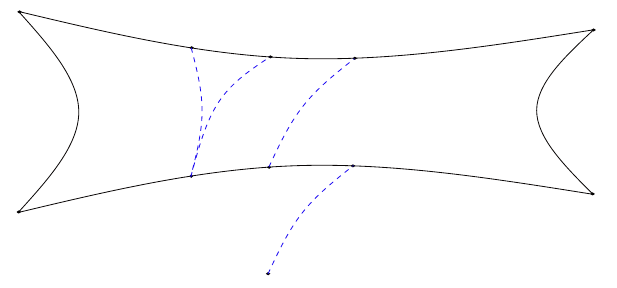\\
  \caption{For Lemma \ref{gromov product is bounded}}\label{sgp}
\end{figure}

\begin{lem}\label{gromov product is bounded}
Let $G$ be a group acting on $S$ by isometries. Let $(x_i)_{i\in \NN}$ be a sequence in $S$ converging to a point in $\partial S$. Fix a basepoint $x_0$ so that $\langle x_i \mid x_j \rangle_{x_0} \longrightarrow \infty$ as $i, j \longrightarrow \infty.$ and an arbitrary $g \in G$. Let $D=\d_S(x_0, gx_0)$. For a given $n \in \NN$ consider the geodesic quadrilateral $x_0, x_n, gx_n, gx_0$.
For $s \in [x_0, x_n]$, if 
\begin{equation}\label{Choice of s 2}
2D+4\delta < \d_S(x_0,s) <\langle x_n \mid  gx_n \rangle_{x_0} -2D -4\delta,
\end{equation}
then $\langle g^{-1}s \mid gs \rangle_s \leq 8\delta$.
\end{lem}

\begin{proof}
First, observe that for $s \in [x_0, x_n]$, if 
\begin{equation}\label{Choice of s}
D+2\delta < \d_S(x_0, s) < \langle x_n \mid gx_n \rangle_{x_0}- 2\delta,
\end{equation}
then $s$ is within $2\delta$ of some point $s' \in [gx_0,gx_n]$. Indeed, any geodesic quadrilateral in $S$ is $2\delta$-slim, so $s$ is within $2\delta$ of a point $s'$ on one of the sides $[x_0,gx_0]$, $[gx_0,gx_n]$, or $[gx_n,x_n]$.
If $s' \in [x_0, gx_0]$, then $$
\d_S(x_0,s) \leq \d_S(x_0,s')+\d_S(s',s) \leq D+2\delta,$$
which contradicts the first inequality in (\ref{Choice of s}).
If $s' \in [x_n, gx_n]$, then $$
\langle x_n \mid gx_n \rangle_{x_0} \leq \d_S(x_0, s') \leq \d_S(x_0,s)+2\delta,
$$ which contradicts the second inequality in (\ref{Choice of s}).

Therefore, if (\ref{Choice of s 2}) is satisfied, then there exists $s' \in [gx_0, gx_n]$ such that $\d_S(s, s') \leq 2\delta$. Consider the point $gs \in [gx_0, gx_n]$.
Note that, since $\d_S(x_0, gx_0) = D$ and $\d_S(s, s') \leq 2\delta$, by the triangle inequality, we obtain
$$
|\d_S(gx_0, s') - \d_S(x_0, s)| \leq D+2\delta.
$$
Thus, we have
$$
|\d_S(gx_0, s') - \d_S(gx_0,gs)| \leq D+2\delta.
$$
Since $gx_0, gs, s'$ are on the same geodesic $[gx_0, gx_n]$, we conclude that $L := \d_S(gs, s') \leq D+ 2\delta$. 
Consider the point $g^{-1}s' \in [x_0, x_n]$. We have $\d_S(s, g^{-1}s')=L$. Let $y$ be the other point on the geodesic $[x_0, x_n]$ at distance $L$ from $s$. Since $L\leq D+ 2\delta$, by (\ref{Choice of s 2}), we obtain that
$$
D+2\delta < \d_S(x_0, y) < \langle x_n \mid gx_n \rangle_{x_0}- 2\delta.
$$
Therefore, $y$ is within $2\delta$ of some point $y' \in [gx_0, gx_n]$.
Note that since $\d_S(s, s')\leq 2\delta$, by the triangle inequality
$|\d_S(y', s') - \d_S(y,s)| \leq 4\delta$.
Thus, we obtain
$$
L-4\delta \leq d_S(y', s') \leq L+4\delta.
$$
Since $\d_S(gs, s') =L$, and $y', gs, s'$ are all on the same geodesic, we get
$\d_S(y', gs) \leq 4\delta$.
We have $\d_S(g^{-1}s, g^{-1}s')=\d_S(s, s') \leq 2\delta$,
and $d_S(gs, y) \leq \d_S(gs, y')+\d_S(y, y') \leq 6\delta$.
So,
$$ 
\begin{aligned}
&\langle g^{-1}s, gs\rangle_s= \frac{1}{2}\left(\d_S(s, g^{-1}s)+\d_S(s, gs)-\d_S(g^{-1}s, gs)\right)\leq \\
&\leq \frac{1}{2}\left(\d_S(s, g^{-1}s')+2\delta
+ \d_S(s,y)+6\delta -\d_S(g^{-1}s', y)+8\delta)\right) \leq\\
&\leq \langle y \mid g^{-1}s'\rangle_s + 8\delta = 8\delta.
\end{aligned}
$$
\end{proof}

\begin{proof}[Proof of Proposition \ref{elements of parabolic subgroups don't move some points far}]
Let $q$ be the unique fixed point of $G$ on $\partial S$ and let $(x_i)_{i\in \NN} \in S$ be a sequence converging to $q$. Fix a basepoint $x_0$ so that $\langle x_i \mid x_j \rangle_{x_0} \longrightarrow \infty$ as $i, j \longrightarrow \infty.$ Now let arbitrary $g_1, g_2 \in G$ be given. Let $D= \max\{\d_S(x_0, g_1x_0), \d_S(x_0, g_2x_0)\}$, and take $n \in \NN$ such that the geodesic segment $[x_0, x_n]$ is sufficiently long to satisfy the condition (\ref{Choice of s 2}) in Lemma \ref{gromov product is bounded}. By Lemma \ref{gromov product is bounded}, we can find $s \in [x_0, x_n]$, such that
$$
\langle g_1^{-1}s \mid g_1s \rangle_s \leq 8\delta \text{ and } \langle g_2^{-1}s \mid g_2s \rangle_s \leq 8\delta.
$$
Since the action of $G$ on $S$ is parabolic, $\mathcal{L}_S(G)=\emptyset$. Therefore, by Lemma \ref{loxodromic criterion},
$$
\d_S(s,g_1s) \leq 56\delta \text{ and } \d_S(s,g_2s) \leq 56\delta. \qedhere
$$
\end{proof}

\begin{lem}\label{Distance formula for loxodromic elements}
Suppose that a group $G$ acts by isometries on a hyperbolic space $S$. Let $g \in \mathcal{L}_S(G)$. Let $L_g$ be the standard ($2, \delta$)-quasi-geodesic axis of $g$. Let $T$ be the translation distance of $g$ along $L_g$. Then there exist $D=D(\delta) \in \NN$ and $N \in \NN$ such that for every $n \geq N$ and for every $s \in S$, we have
$$
2\d_S(s, L_g)+\frac{n}{2}T - D \leq \d_S(s, g^ns) \leq 2\d_S(s, L_g)+2nT + D.
$$
\end{lem}

\begin{proof}
First, choose $N \in \NN$ such that for any $x \in L_g$ and $n \geq \NN$ the distance $\d_S(x, g^nx) \geq 2M(\delta)+14\delta$, where $M$ is a constant from Lemma \ref{projection distanse formula}, applied to $L_g$. Take any $s \in S$. Let $\overline s$ be any nearest point projection of $s$ on $L_g$. Then $g^n\overline{s} \in L_g$ and $g^n\overline{s}$ is the nearest-point projection of $g^ns$ onto $L_g$.
   
\begin{equation} \label{global estimate}
\begin{aligned}
&\d_S(s, g^ns) \geq \d_S(s, \overline{s})+\d_S(\overline{s}, g^n\overline{s})+\d_S(g^n\overline{s}, g^ns)-24 \delta-4M(\delta)\\
&\geq 2\d_S(s, L_g)+\d_S(\overline{s}, g^n\overline{s})-24 \delta-4M(\delta)-2.
\end{aligned}
\end{equation}
Let $\gamma$ be the segment of $L_g$ from $\overline{s}$ to $g^n\overline{s}$. Since $L_g$ is a $(2, \delta)$-quasi-geodesic and $T$ is the translation distance of $g$ along $L_g$, we have
\begin{equation} \label{middle estimate}
\frac{n}{2}T -\delta =\frac{1}{2}l(\gamma)-\delta \leq \d_S(\overline{s}, g^n\overline{s}) \leq 2l(\gamma)+\delta = 2nT + \delta.
\end{equation}
Combining (\ref{global estimate}), (\ref{middle estimate}), and the triangle inequality, we obtain 
$$
2\d_S(s, L_g)+\frac{n}{2}T - D \leq \d_S(s, g^ns) \leq 2\d_S(s, L_g)+2nT + D,
$$
where $D= 25\delta + 4M(\delta)+2$. 
\end{proof}

\subsection{Minimal actions  on trees}

We will use the facts established in this subsection to obtain the results about groups acting on simplicial trees from \cite{BM} and \cite{FMMS} as corollaries of Theorem \ref{main_theorem} and Theorem \ref{thm about conjugate orbits}.

\begin{prop}\cite[Proposition 7.5]{Bas}\label{union of loxodromic axes}
Suppose that a group $G$ acts on a simplicial tree $T$ by isometries.
If $G$ contains a loxodromic element, then there is a unique minimal nonempty $G$-invariant subtree $L \subseteq T$, which is the union of the axes of all loxodromic elements. 
\end{prop}

The next lemma shows that if a group $G$, acting on a tree $T$, contains loxodromic elements, then the minimality of the action is almost equivalent to $\Lambda_T(G) = \partial T$, in the sense that every nontrivial invariant subtree is ``large''.

\begin{lem}\label{limit set is the whole boundary}
Suppose that the group $G$ acts on a simplicial tree $T$ by isometries, and $\mathcal{L}_T(G) \neq \emptyset$. If the action of $G$ on $T$ is minimal, then $\Lambda_T(G) =\partial T$.
Moreover, without the minimality assumption, the following conditions are equivalent:
\begin{enumerate}
\item $\Lambda_T(G) = \partial T$;
\item For every $G$-invariant subtree $S \subseteq T$, we have $\partial T = \partial S$.
\end{enumerate}
\end{lem} 
\begin{proof}
By Proposition \ref{union of loxodromic axes}, the union of the axes of all loxodromic elements $L\subseteq T$ is the unique nonempty $G$-invariant subtree. Since the action of $G$ on $T$ is minimal,  $L=T$. By the construction of $L$, the set of fixed points of loxodromic elements of $G$ is dense in $\partial L$. Therefore, $\Lambda_L(G)=\partial L=\partial T$. 

Now, assume that $\Lambda_T(G) = \partial T$. Let $S \subseteq T$ be a $G$-invariant subtree. Note that $\Lambda_T(G) \subseteq \partial S$ since $\Lambda_T(G)$ does not depend on a chosen point $t \in T$ and we can choose $t \in S$. Then we have $\partial T \subseteq \partial S$, so $\partial T =\partial S$. Conversely, assume that for every $G$-invariant subtree $S \subseteq T$, we have $\partial S = \partial T$. Let $L$ be the
union of the axes of all loxodromic elements. By Proposition \ref{union of loxodromic axes}, $L$ is a $G$-invariant subtree of $T$, and $\Lambda_L(G)=\partial L$ by the construction of $L$. Since the limit set does not depend on a chosen point $t \in T$, we have $\Lambda_T(G)=\Lambda_L(G)=\partial L = \partial T$. 
\end{proof}

\begin{lem}\label{fixed point in a tree}
 Every elliptic subgroup of a group $G$ acting on any tree $T$ by isometries either fixes a vertex or a midpoint of an edge.
\end{lem}
\begin{proof}
By \cite[Chapter 3.b, Proposition 9]{HV}, every elliptic subgroup of $G$ fixes some point of $T$, and, by \cite[Proposition 7.2]{Bas}, this point should be either a vertex or a midpoint of an edge.
\end{proof}

\subsection{Elliptic radical} \label{elliptic radical}

\begin{defn}\cite{Osi17}
Given an action of a group $G$ by isometries on a hyperbolic space $S$, define the elliptic radical of $G$ with respect to this action by
$$
E_S(G)=\{g \in G \mid g s=s  \; \; \forall s \in \Lambda_S(G)\}.
$$
\end{defn}

Note that this definition makes sense for arbitrary actions on hyperbolic spaces. For general type actions, the following proposition provides an equivalent characterization.

\begin{prop}\cite[Proposition 3.4]{Osi17}\label{elliptic radical is maximal}
Suppose that a group $G$ admits a general type action on a hyperbolic space $S$. Then $E_S(G)$ is the unique maximal elliptic normal subgroup of $G$.
\end{prop}

Recall that the {\it finite radical} $K(G)$ of an acylindrically hyperbolic group $G$ is the unique maximal finite normal subgroup of $G$ (see \cite[Thm 6.14]{DGO}). 

\begin{lem}\cite[Lemma 3.7]{BH}\label{E(G)=K(G)}
If the action of a group $G$ on a hyperbolic space $S$ is of general type and acylindrical, then $E_S(G)=K(G)$.
\end{lem}

The next lemma shows that the property of being faithful weakly hyperbolic and the property of being lim-free pass to normal subgroups; thus, all the results in this work pass to normal subgroups.

\begin{lem}\label{lim-free passes to normal subgroups}
Let $H$ be a non-trivial normal subgroup of a faithful weakly hyperbolic group $G$. Then $H$ is also a faithful weakly hyperbolic group. If $G$ is lim-free, then $H$ is also lim-free.
\end{lem}
\begin{proof}
Suppose that the action of $G$ on a hyperbolic space $S$ is of general type and the induced action of $G$ on $\Lambda_S(G)$ is faithful. The subgroup $H$ is normal in $G$, so by Lemma \ref{dichotomy}, either the action of $H$ on $S$ is elliptic or it is of general type; in the latter case, $\Lambda_S(H)=\Lambda_S(G)$. Assume that $H$ is elliptic. Since $H$ is normal, $ H\leq E_S(G)$ by the maximality of $E_S(G)$, and, thus, is trivial because the action of $ G$ on $\Lambda_S(G)$ is faithful, which is a contradiction. Therefore, the action of $H$ on $S$ is of general type, and the induced action of $H$ on $\Lambda_S(H)=\Lambda_S(G)$ is faithful. If the action of $G$ on $\Lambda_S(G)$ is topologically free, so is the action of $H$ on $\Lambda_S(H)$.  
\end{proof}

We finish this subsection by proving Proposition \ref{faithful weakly hyperbolic quotient}, which allows us to apply the main results of this work to some quotients of weakly hyperbolic groups that are not faithful weakly hyperbolic. 

\begin{proof}[Proof of Proposition \ref{faithful weakly hyperbolic quotient}]
We need to show that there exists a hyperbolic space $S'$ such that the action of $G/E_S(G)$ on $S'$ is also of general type and the induced action on $\Lambda_{S'}(G/E_S(G))$ is faithful.

We will use the following lemma. The proof appears in \cite{BFGS}, though the statement is different.

\begin{lem} \cite[Lemma 4.10]{BFGS} \label{construction of a new hyp space} Suppose that $N$ is a normal subgroup of $G$, and $G$ acts on a hyperbolic space $S$ such that the action of $N$ is elliptic. Then there is a $G$-invariant  $Y \subset S$, a hyperbolic graph $S'$, a homomorphism $G / N \rightarrow$ Isom $S'$, and a quasi-isometry $\varphi: Y \rightarrow S'$ that is $G$-equivariant.
\end{lem}

We apply Lemma \ref{construction of a new hyp space} to $E_S(G)$. Since $Y \subset S$ is $G$-invariant, it contains orbits $\langle g \rangle y$ and $\langle f \rangle y$ of two elements $f, g \in \mathcal{L}_S(G)$ with disjoint fixed points, for some $y \in Y$. Since $\varphi: Y \rightarrow S^{\prime}$ is a quasi-isometry, the images of orbits $\varphi(\langle g \rangle y)$ and $\varphi(\langle f \rangle y)$ bi-infinitely diverge. Thus, $gE_S(G), fE_S(G) \in \mathcal{L}_{S'}(G/ E_S(G))$ with disjoint fixed points because $\varphi$ is $G$-equivariant and orbits of $E_S(G)$ are bounded. Therefore, the action of $G/E_S(G)$ on $S'$ is of general type. The preimage of $E_{S'}(G/E_S(G))$ under the homomorphism $G \rightarrow G/E_S(G)$ is a normal subgroup of $G$ and, since $\varphi: Y \rightarrow S^{\prime}$ is a $G$-equivariant quasi-isometry, it is also elliptic. By the maximality of $E_S(G)$ (Proposition \ref{elliptic radical is maximal}), the preimage of $E_{S'}(G/E_S(G))$ is contained in the kernel of the homomorphism. Thus, we have $E_{S'}(G/E_S(G)) =1$ and the action of $G/E_S(G)$ on $\Lambda_{S'}(G/E_S(G))$ is faithful.
\end{proof}

\subsection{Micro-supported actions}\label{micro-supported actions}

\begin{defn}
Suppose that a group $G$ acts on a set $X$. The {\it support of an element $g \in G$} is the set $\supp(g)=\{ x \in X \mid gx \neq x \}$. 
\end{defn}

\begin{defn}
Let $G$ be a group that acts on a set $X$. Given a subset $U \subseteq X$, the rigid stabilizer of $U$ is the subgroup
$$
\rist_G(U):=\{g \in G \mid g x=x \quad \forall x \in X\setminus U\} .
$$
\end{defn}

\begin{defn}
An action of a group $G$ by homeomorphisms on a topological space $X$ is called {\it micro-supported} if, for every nonempty open subset $U \subseteq X$, there is an element of $G$ acting non-trivially on $U$ and trivially on the complement $X \setminus U$. If the action of $G$ on $X$ is faithful, this is the same as having $\rist_G(U) \neq 1$ for every nonempty open subset $U \subseteq X$.
\end{defn}

The following fact is standard and known to experts (for example, see the proof of \cite[Proposition H]{CRW}).

\begin{lem}\label{non-abelian rigid stabilizers}
Let $G$ be a subgroup of the homeomorphism group of a Hausdorff topological space $X$. If the $G$-action is micro-supported, then $\rist_G(U)$ is non-abelian for every nonempty open subset $U \subseteq X$.
\end{lem}

\begin{lem}\label{top.k-transitive or micro-supported}
Suppose that the action of a group $G$ on a topological space $X$ by homeomorphisms is faithful and micro-supported. Then it is either topologically $k$-transitive for all $k \in \NN$ or $G$ satisfies a non-trivial mixed identity.
\end{lem}

\begin{proof}
    Suppose that the action of $G$ on $X$ is not topologically $k$-transitive for some $k \in \NN$.
Then there exist nonempty open subsets $A_1, \ldots, A_k$ and $B_1, \ldots, B_k$ of $X$ such that for every $g \in G$ there exists  $i \in \{1, \ldots, k \}$ with  $A_i \cap gB_i = \emptyset$.  Let us fix an arbitrary $\widehat{g} \in G$ and, without loss of generality, assume that $A_1 \cap \widehat{g}B_1 = \emptyset$.
    
     Since the action of $G$ on $X$ is micro-supported, we can find non-trivial elements $a_1, \ldots, a_k$ and $b_1, \ldots, b_k \in G$, such that $\supp(a_i) \subseteq A_i$ and $\supp(b_i) \subseteq B_i$ for each $i \in \{1, \ldots, k \}$.
    Then for any point $x \notin \widehat{g}^{-1}A_1$ we have $\widehat{g}^{-1}a_1\widehat{g}x = \widehat{g}^{-1}\widehat{g}x = x$. Therefore, $\supp(a_1^{\widehat{g}}) \subseteq \widehat{g}^{-1}A_1$ and we obtain $\supp(a_1^{\widehat{g}}) \cap \supp(b_1) = \emptyset$. We conclude that $a_1^{\widehat{g}}$ and $b_1$ commute, since they have disjoint support. 

This argument shows that for an arbitrary $g \in G$ we have
    \begin{equation} \label{commutation}
        [a_1^g, b_1]=1 \lor [a_2^g, b_2]=1 \lor \ldots \lor [a_k^g, b_k]=1.
    \end{equation}
    Denote by $c_i : = [a_i^t, b_i] \in G*\l t \rangle$.  Since $a_1, \ldots, a_k$ and $b_1, \ldots, b_k$ are non-trivial, the Normal Form Theorem (see, for example, \cite[Theorem 1.2, Chapter 4]{LS}) implies that each $c_i$ is a non-trivial element of $G* \l t \r $. 
    Let $x_1, \dots, x_{k-1}$ be elements of a free group $F(x_1, \dots, x_{k-1})$,
    and let $w_k$ be the iterated commutator of the elements $c_1, c_2^{x_1}\ldots, c_k^{x_{k-1}}$:
    $$
        w_k=[c_k^{x_{k-1}},[c_{k-2}^{x_{k-2}},[\ldots[c_{1},c_2^{x_1}]\ldots].
    $$
   Clearly, $w_k$ is a non-trivial element of $G* F(t,x_1,x_2, \dots, x_{k-1})$. By (\ref{commutation}), $G$ satisfies the mixed identity $w_k=1$.
    \end{proof}

\begin{defn}
Suppose that a group $G$ acts on a topological space $X$ by homeomorphisms. Let $H \leq G$. A subset $V$ of $X$ is called {\it $H$-compressible} if it is nonempty and if for any nonempty open subset $U$, there exists $h \in H$ such that $hV \subseteq U$. We say that the action of $G$ is {\it compressible} if a nonempty $G$-compressible open subset exists and {\it fully compressible} if every non-dense open set is $G$-compressible.
\end{defn}

\begin{lem}\label{Gen type actions admit compressible open sets}
Suppose that the action of a group $G$ on a hyperbolic space $S$ is quasi-parabolic, and let $\Fix_{\partial S}(G)=\{q\}$. Then every nonempty open subset $V \subset \Lambda_S(G)$ with $q \notin \overline{V}$ is $G$-compressible.  More precisely, given any nonempty open subset $V \subset \Lambda_S(G)$ with $q \notin \overline{V}$  and any nonempty open $U \subseteq \Lambda_S(G)$, there exists $l \in \mathcal{L}_S(G)$ such that $lV \subset U$. If the action of $G$ on $S$ is of general type, then the action of $G$ on $\Lambda_S(G)$ is fully compressible. More precisely, given any nonempty non-dense open subset $V \subset \Lambda_S(G)$ and any nonempty open $U \subseteq \Lambda_S(G)$, there exists $l \in \mathcal{L}_S(G)$ such that $lV \subset U$.
\end{lem}
\begin{proof}
Let $V \subset \Lambda_S(G)$ be a nonempty open subset with $q \notin \overline{V}$. Take any nonempty open subset $U \subseteq \Lambda_S(G)$.
Since $q \notin \overline{V}$, there exist two disjoint open subsets $F^{+} \subset U$ and $F^{-} \subset \Lambda_S(G) \setminus \overline{V}$ such that $q \in F^-$. By Lemma \ref{set of fixed points of loxodromic elements is dense}, there exists $f \in \mathcal{L}_S(G)$ with $f^+ \in F^{+}$ and $f^- = q$. By Lemma \ref{north-south} there exists $k \in \NN$ such that $f^k(\Lambda_S(G) \setminus F^-) \subseteq F^+$.
Taking $l=f^k$, we obtain 
$$lV =f^kV \subset f^k(\Lambda_S(G) \setminus F^-) \subset F^+ \subset U.$$
The proof of the general type case is similar.
\end{proof}

\begin{lem}\label{Not_top_free actions are micro-supported}
Suppose that the action of a group $G$ on a hyperbolic space $S$ is of general type. If the induced action on $\Lambda_S(G)$ is faithful and not topologically free, it is micro-supported. 
\end{lem}

\begin{proof}
Let $U$ be an arbitrary nonempty open subset of $\Lambda_S(G)$. Since the action of $G$ on $\Lambda_S(G)$ is not topologically free, there exists $a \in G \setminus\{1\}$ such that $\Int(\Fix_{\Lambda_S(G)}(\langle a \rangle)) \neq \emptyset$. Therefore, $V=\Lambda_S(G) \setminus \Fix_{\Lambda_S(G)}(\langle a \rangle)$ is a non-dense open subset of $\Lambda_S(G)$. Since the action of $G$ on $\Lambda_S(G)$ is faithful, $V$ is nonempty and $a \in \rist_G(V)$ acts nontrivially on $V$. Then, by Lemma \ref{Gen type actions admit compressible open sets}, there exists $l \in \mathcal{L}_S(G)$ such that $lV \subset U$. Therefore, $l\rist_G(V)l^{-1} \leq \rist_G(U)$. 
\end{proof}

\begin{cor}\label{infinite direct products}
Let $G$ be a faithful weakly hyperbolic group that is not lim-free. Then $G$ contains infinite direct sums of infinite non-abelian subgroups. 
\end{cor}
\begin{proof}
Let $S$ be a hyperbolic space such that the action of $G$ on $S$ is of general type, and the induced action on $\Lambda_S(G)$ is faithful and not topologically free. Take an infinite set of pairwise disjoint non-empty open subsets $V_1, \dots, V_n, \dots \subset \Lambda_S(G)$. By Lemma \ref{Not_top_free actions are micro-supported}, rigid stabilizers $\rist_G(V_i)$ are nonempty, and, by Lemma \ref{non-abelian rigid stabilizers}, they are non-abelian for all $i \in \NN$. Elements in $\rist_G(V_i)$ and $\rist_G(V_j)$ have disjoint support for $i \neq j$, therefore they commute. Thus, $\rist_G(V_i), i \in \NN$, form a direct sum. Every nonempty open subset of $\Lambda_S(G)$ contains infinitely many disjoint nonempty open subsets. Therefore, rigid stabilizers $\rist_G(V_i)$ are infinite.
\end{proof}

We already know that if the $G$-action on a Hausdorff topological space $X$ is micro-supported, then the rigid stabilizer of a nonempty open subset $U \subset X$ is non-abelian. If, in addition, $U$ is $G$-compressible, we can say more.  

\begin{lem}\label{wreath products in rigid stabilizers}
Let $G$ be a subgroup of the homeomorphism group of a Hausdorff topological space $X$, and $U \subset X$ be a nonempty open $G$-compressible subset. Either $\rist_G(U)$ is torsion-free, or there exists a finite $k \geq 2$ such that 
$\rist_G(U) \wr \ZZ_k < \rist_G(U)$.  
\end{lem}
\begin{proof}
Suppose that a non-identity element $g \in \rist_G(U)$ has finite order. Take $x \in U$ such that the orbit of $x$ under the action of $\l g \r$ has the maximum size $k$. Since the action of $G$ on $X$ is faithful, $k \geq 2$ and divides the order of $g$. Consider the distinct points $gx, \dots, g^{k-1}x$. Since $X$ is Hausdorff, there exist disjoint open subsets $V_0, \dots, V_{k-1}$, such that $g^ix \in V_i$ for $i \in \{0, \dots, k-1\}$. Consider the subset 
$$
V=\bigcap_{i=0}^{k-1} g^{-i}V_i.
$$
It is open and nonempty since $x \in V$. By the maximality of the orbit size, a quotient of $\l g \r$ isomorphic to $\ZZ^k$ acts on $V$ by permutations, and $V, gV, \dots, g^{k-1}V$ are pairwise disjoint. 

Since $U$ is $G$-compressible, there exists $f \in G$ such that $fU \subset V$. Therefore, 
$$
g^if\rist_G(U)(g^if)^{-1} \leq \rist_G(g^iV) \leq \rist_G(U), \; \; i \in \{0, \dots, k-1\}.
$$ 

The subgroups $g^if\rist_G(U)(g^if)^{-1}$ commute, since their elements have disjoint support, and the subgroup $\l g \r < \rist_G(U)$ acts on the direct product 
$$
\prod_{i=0}^{k-1}g^if\rist_G(U)(g^if)^{-1}
$$
permuting the terms.
Thus, $\rist_G(U) \wr \ZZ_k < \rist_G(U)$.
\end{proof}

\section{Boundary dynamics and algebraic properties}

\subsection{The cross-ratio and topological transitivity of actions on the limit set}

In this subsection, we develop a technical tool that is later used in the proof of Proposition \ref{non-elementary action}. We also use it to prove that the boundary action arising from a general-type action is never topologically $4$-transitive (see Proposition \ref{not 4-transitivity}). A similar notion in the case of proper hyperbolic spaces was considered in \cite{Paul}.

\begin{defn}
Let $a, b, c, d$ be four distinct points of $\partial S$ and let $o \in S$. Then the {\it cross-ratio} $[a, b, c, d]_o$ is the real number defined by
\begin{equation}\label{useful definition}
[a, b, c, d]_o = \l a|b\r_o+\l c|d\r_o-\l a|d\r_o-\l b|c\r_o.
\end{equation} 

\end{defn}

Since the Gromov product of distinct points of the boundary is finite and is preserved by isometries of $S$, we can see that the cross-ratio of distinct points of $\partial S$ is also finite, and is also preserved by isometries of $S$, i.e., if $g$ is an isometry of $S$, then 
$$
[ga, gb, gc, gd]_{go}=[a, b, c, d]_o.
$$
For $x,y \in \RR, \varepsilon \geq 0$, we say that $x \approx_{\varepsilon}y$ if $\left|x - y\right|\leq \varepsilon$.
\begin{rem}\label{technical}
Note that for arbitrary bounded sequences $(p_i), (q_i) \in \RR$, if 
$$\liminf_{i \rightarrow \infty} p_i \approx_{\varepsilon} \limsup_{i \rightarrow \infty} p_i \text{ and } \liminf_{i \rightarrow \infty} q_i \approx_{\varepsilon} \limsup_{i \rightarrow \infty} q_i, \text{ then }
$$
$$
\liminf_{i \rightarrow \infty} p_i+\liminf_{i \rightarrow \infty} q_i \approx_{2\varepsilon} \liminf_{i \rightarrow \infty} \left(p_i+q_i\right).
$$
\end{rem}
The next lemma shows that the cross-ratio does not depend much on the choice of the base point $o \in S$.
\begin{lem}\label{cross-ratio doesn't depend on o}
Let $a,b,c,d$ be distinct points of $\partial S$ and let $o, o' \in S$. Then
$$
[a,b,c,d]_o \approx_{48\delta} [a,b,c,d]_{o'}.
$$
\end{lem}
\begin{proof}
Let $(x_i), (y_i), (z_i)$, and $(t_i)$ be arbitrary sequences converging to $a, b, c$, and $d$, respectively. Applying Lemma \ref{lemma 2.2.2}(1), we get
\begin{equation}\label{sum of limits}
\begin{aligned}
&[a,b,c,d]_o \approx_{8\delta}
\liminf _{i \rightarrow \infty}\l x_i| y_i\r_o + \liminf _{i \rightarrow \infty}\l z_i| t_i\r_o - \liminf _{i \rightarrow \infty}\l x_i| t_i\r_o - \liminf _{i \rightarrow \infty}\l y_i| z_i\r_o.
\end{aligned}
\end{equation}
Also, since points $a, b, c, d$ are distinct, $\l a|b\r_o, \l c|d\r_o, \l a|d\r_o$, and $\l b|c\r_o$ are finite; therefore, sequences $(\l x_i| y_i\r_o), (\l z_i| t_i\r_o), (\l x_i| t_i\r_o)$ and $(\l y_i| z_i\r_o)$ are bounded.
Thus, by (\ref{sum of limits}), Lemma \ref{lemma 2.2.2}, Remark \ref{technical}, and the definition of the Gromov product, we obtain 
$$
\begin{aligned}
&[a,b,c,d]_o\approx_{24\delta} \liminf _{i \rightarrow \infty}\left(\l x_i| y_i\r_o+ \l z_i| t_i\r_o - \l x_i| t_i\r_o - \l y_i| z_i\r_o\right)=\\
&=1/2\liminf _{i \rightarrow \infty}(\d_S(x_i,t_i)+\d_S(y_i,z_i)-\d_S(x_i,y_i)-\d_S(z_i,t_i)).
\end{aligned}
$$
Therefore, the cross-ratio $[a,b,c,d]_o$ is independent of the choice of $o \in S$ up to $48\delta$.
\end{proof}
The next lemma is technical and will be used mainly in the proof of stability of the cross-ratio.
\begin{lem}\label{Gromov product is small}
Let $o\in S$, $a,b \in \partial S$ and $\l a|b\r_o=R < \infty$. Then for arbitrary $a' \in \mathcal U_o(a, R+2\delta)$ and $b' \in \mathcal U_o(b, R+2\delta)$, we have
\begin{equation}\label{a'|b}
\max\{\l a'|b\r_o, \l a|b'\r_o, \l a'|b'\r_o\} \leq \min\{ \l a'|a\r_o, \l b'|b\r_o\}
\end{equation}

\end{lem}
\begin{proof}
Assume that $\l a'|b\r_o > \l a'|a\r_o$. Then by Lemma \ref{lemma 2.2.2}(2), we have 
$$
R=\l a|b\r_o \geq \min\left\{\l a|a'\r_o,\l a'|b\r_o\right\} - \delta=\l a|a'\r_o - \delta > R+\delta,
$$
which gives a contradiction. The inequality $\l a'|b\r_o \leq \l b'|b\r_o$ and inequalities for $\l a|b'\r_o$ follow by the same reasoning.

Now assume that $\l a'|b'\r_o > \l b'|b\r_o$. Then by Lemma \ref{lemma 2.2.2}(2), we have
$$
\begin{aligned}
&R=\l a|b\r_o \geq \min\left\{\l a|a'\r_o,\l a'|b\r_o\right\} - \delta \geq \min\left\{\l a|a'\r_o,\min\left\{\l a'|b'\r_o,\l b'|b\r_o\right\} - \delta \right\} - \delta=\\
&=\min\left\{\l a|a'\r_o, \l b'|b\r_o  - \delta \right\} - \delta > R,\\
\end{aligned}
$$
which gives a contradiction. Similarly, $\l a'|b'\r_o \leq \l a'|a\r_o$.
\end{proof}

\begin{lem}\label{cross-ratio is arbirtary large}
Let $o \in S$ and $|\Lambda_S(G)|=\infty$. If $\mathcal{O}$ is a dense subset of $\Lambda_S(G)$, then there exist distinct $a,b,c,d \in \mathcal{O}$ such that $[a,b,c,d]_o$ is arbitrarily large. 
\end{lem}
\begin{proof}
Let $M >0$, $o\in S$, and $a, c \in \mathcal{O}$ be given. Let $\l a|c \r_o=K < \infty$. Take $N=M/2+2(K+\delta)$.
If $|\Lambda_S(G)|=\infty$, then, by  Lemma \ref{set of fixed points of loxodromic elements is dense}, $\mathcal{H}_S(G)$ is dense in $\Lambda_S(G)$; in particular, $\Lambda_S(G)$ is Hausdorff and perfect. Since $\mathcal{O}$ is dense in $\Lambda_S(G)$, it is not discrete and has no isolated points. Thus, there exist $b \in (\mathcal U_o(a,N)\cap \mathcal{O})\setminus\{a\}$ and $d \in (\mathcal U_o(c,N)\cap \mathcal{O})\setminus\{c\}$. Then, by Lemma \ref{Gromov product is small} and Lemma \ref{lemma 2.2.2}(2), we have
$$
K=\l a|c\r_o \geq \min\left\{\l a|b\r_o,\l b|c\r_o\right\} - \delta = \l b|c\r_o - \delta.
$$
Therefore, $\l b|c\r_o \leq K+\delta$. Similarly, $\l a|d\r_o \leq K+\delta$. Thus, we obtain
$$
\l a|b\r_o + \l c|d\r_o -\l a|d\r_o -\l b|c\r_o > 2N - 2(K+\delta) =M. \qedhere
$$
\end{proof}
The next lemma shows that the cross-ratio is stable, i.e., it does not change much if we slightly move the points of the boundary. 
\begin{lem}\label{cross-ratio doesn't change much}
Let $ a,b, c$, and $d$ be distinct points of $\partial S$, and let $o \in S$. Define $R=\max\left\{\l a|b\r_o, \l c|d\r_o, \l a|d\r_o, \l b|c\r_o \right\}+2\delta$. 
If $a' \in \mathcal U_o(a, R)$, $b' \in \mathcal U_o(b, R)$, $c' \in \mathcal U_o(c, R)$, and  $d' \in \mathcal U_o(d, R)$, then 
$$
[a, b, c, d]_o \approx_{8\delta} [a', b', c', d']_o.
$$
\end{lem} 
\begin{proof}
By Lemma \ref{lemma 2.2.2}(2), we have
$$
\l a'|b'\r_o \geq \min\left\{\l a'|a\r_o,\min\left\{\l a|b\r_o, \l b|b'\r_o\right\} - \delta\right\} - \delta=\l a|b\r_o - 2\delta.
$$
Similarly, by Lemma \ref{lemma 2.2.2}(2) and Lemma \ref{Gromov product is small}, we obtain $\l a|b\r_o \geq \l a'|b'\r_o - 2\delta$, and thus, $\l a|b\r_o \approx_{2\delta} \l a'|b'\r_o$.
Using the same argument, we get
$$
\l c|d\r_o \approx_{2\delta} \l c'|d'\r_o, \; \l a|d\r_o \approx_{2\delta} \l a'|d'\r_o, \text{ and } \l b|c\r_o \approx_{2\delta} \l b'|c'\r_o.
$$
This gives us the desired inequality $\left|[a,b,c,d]_o - [a',b',c',d']_o\right|\leq 8\delta$.
\end{proof}

\begin{prop}\label{not 4-transitivity}
 Let $G$ be a group admitting a non-elementary action on a hyperbolic space $S$, and let $\mathcal{O}$ be a $G$-invariant subset of $\Lambda_S(G)$.
\begin{enumerate}
    \item[(a)] If the action of $G$ on $S$ is of general type, then the induced action of $G$ on $\mathcal{O}$ is not topologically $4$-transitive.
    \item[(b)] If the action of $G$ on $S$ is quasi-parabolic, then the induced action of $G$ on $\mathcal{O}$ is not topologically $3$-transitive.
\end{enumerate}
\end{prop}

\begin{proof}
 By Lemma \ref{set of fixed points of loxodromic elements is dense}, the set of fixed points of elements of $\mathcal{L}_S(G)$ is dense in $\Lambda_S(G)$. Thus, since $\mathcal{O}$ is $G$-invariant, by Lemma \ref{north-south}, it is also dense in $\Lambda_S(G)$.
\begin{enumerate}
\item[(a)]
Let $o \in S$. By Lemma \ref{cross-ratio is arbirtary large}, the cross-ratio of distinct points of $\mathcal{O}$ can be arbitrarily large. Therefore, there exist eight distinct points $a, b, c, d, a', b', c', d' \in \mathcal{O}$ such that 
$$\left|[a, b, c, d]_o - [a', b', c', d']_o\right| \geq 65\delta.$$
Define
$$
\begin{aligned}
&R_1=\max\left\{\l a|b\r_o, \l c|d\r_o, \l a|d\r_o, \l b|c\r_o \right\}+2\delta \text{ and }\\
& R_2=\max\left\{\l a'|b'\r_o, \l c'|d'\r_o, \l a'|d'\r_o, \l b'|c'\r_o \right\}+2\delta.
\end{aligned}
$$ 
By way of contradiction, assume that the action of $G$ on $\mathcal{O}$ is topologically $4$-transitive. Then there exists $g \in G$ and points 
$$\widehat{a} \in  \mathcal U_o(a', R_2) \cap \mathcal{O}, \;  \widehat{b} \in  \mathcal U_o(b', R_2) \cap \mathcal{O}, \;  \widehat{c} \in  \mathcal U_o(c', R_2) \cap \mathcal{O}, \;  \widehat{d} \in  \mathcal U_o(d', R_2) \cap \mathcal{O}$$
such that 
$$g\widehat{a} \in \mathcal U_o(a, R_1) \cap \mathcal{O}, \; g\widehat{b} \in \mathcal U_o(b, R_1) \cap \mathcal{O}, \;  g\widehat{c} \in \mathcal U_o(c, R_1) \cap \mathcal{O}, \text{ and } g\widehat{d} \in \mathcal U_o(d, R_1) \cap \mathcal{O}.$$
Using Lemma \ref{cross-ratio doesn't change much} and Lemma \ref{cross-ratio doesn't depend on o}, we obtain
$$
\begin{aligned}
&65\delta \leq \left|[a, b, c, d]_o - [a', b', c', d']_o\right| \leq \left|[a, b, c, d]_o - [\widehat{a}, \widehat{b}, \widehat{c}, \widehat{d}]_o\right|+\left|[\widehat{a}, \widehat{b}, \widehat{c}, \widehat{d}]_o-[a', b', c', d']_o\right| \leq \\
& \leq \left|[a, b, c, d]_o - [g\widehat{a}, g\widehat{b}, g\widehat{c}, g\widehat{d}]_{o}\right|+\left|[g\widehat{a}, g\widehat{b}, g\widehat{c}, g\widehat{d}]_{o}-[g\widehat{a}, g\widehat{b}, g\widehat{c}, g\widehat{d}]_{go}\right| +\\
&+ \left|[\widehat{a}, \widehat{b}, \widehat{c}, \widehat{d}]_o-[a', b', c', d']_o\right|\leq 64\delta,
\end{aligned}
$$ which is a contradiction.
\item[(b)] Let $q$ be a global fixed point in $\Lambda_S(G)$. Let $o\in S$. Again, using Lemma \ref{cross-ratio is arbirtary large}, we find distinct points $a, b, c, a', b', c' \in \mathcal{O}$ such that
$$\left|[a,b,c,q]_o - [a',b',c',q]_o\right| \geq 65\delta.$$ The rest of the proof is similar to part $(a)$, except that we use a fixed point $q$ instead of $d$ and $d'$.\end{enumerate} \end{proof}

\begin{rem}
M. Gromov \cite[8.2.H]{Gro87} observed that if a group $G$ admits a non-elementary action on a hyperbolic space $S$, then the induced action of $G$ on $\Lambda_S(G)$ is always topologically transitive. If, in addition, the action of $G$ on $S$ is of general type, then the induced action on $\Lambda_S(G)$ is always topologically $2$-transitive.
\end{rem}

We will later use the following easy observation.
 
\begin{lem}\label{transitivity and topological transitivity}
If an action of a group $G$ on a dense subset $\mathcal{O}$ of a Hausdorff topological space $X$ without isolated points is $k$-transitive, then it is topologically $k$-transitive.
\end{lem}
\begin{proof}
Let $U_1, \dots, U_k$ and $V_1, \dots, V_k$ be nonempty open subsets of $\mathcal{O}$. Being a dense subset of a Hausdorff topological space without isolated points, $\mathcal{O}$ is itself Hausdorff without isolated points. Therefore, we can find pairwise distinct $x_1, \dots, x_k$ and pairwise distinct $y_1, \dots, y_k $ such that $x_i \in U_i$ and $y_i \in V_i$ for every $i=1, \dots, k$. Since the action of $G$ on $\mathcal{O}$ is $k$-transitive, there exists $g \in G$ such that $gy_i=x_i$ for all $i=1, \dots, k$; in particular, $gV_i \cap U_i \neq \emptyset$ for every $i=1, \dots, k$.
\end{proof}

\subsection{Boundary dynamics and mixed identities}\label{boundary dynamics and mixed identities}

\begin{thm}\label{real main theorem}
Suppose that an action of a group $G$ on a hyperbolic space $S$ is of general type and the induced action on $\Lambda_S(G)$ is faithful. Then the following are equivalent:
\begin{enumerate}
    \item[(1)] $G$ is MIF.
     \item[(2)] The induced action of $G$ on $\Lambda_S(G)$ is topologically free.
\end{enumerate}
\end{thm}

\begin{proof}
$(1) \Rightarrow (2)$.
    By contrapositive, suppose that the action of $G$ on $\Lambda_S(G)$ is not topologically free. Then by Lemma \ref{Not_top_free actions are micro-supported}, it is micro-supported. The action of $G$ on $\Lambda_S(G)$ is not topologically 4-transitive by Proposition \ref{not 4-transitivity}. Therefore, by Lemma \ref{top.k-transitive or micro-supported}, $G$ satisfies a nontrivial mixed identity.

\begin{figure}
 \centering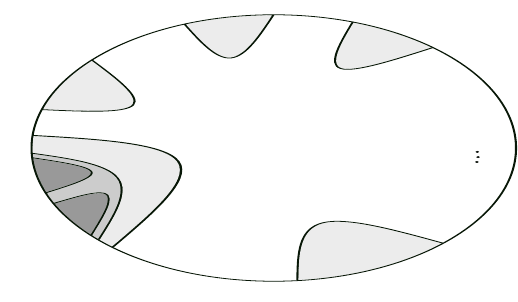\\
  \caption{For Theorem \ref{real main theorem}}\label{fig}
\end{figure}
$(2) \Rightarrow (1)$.
Assume that the action of $G$ on $\Lambda_S(G)$ is topologically free. Since the action of $G$ on $S$ is of general type, $G$ contains a free subgroup and does not satisfy any nontrivial identity. By \cite[Rem. 5.1]{OH}, $G$ satisfies some non-trivial mixed identity $w \in G * F_n$ if and only if it satisfies a non-trivial mixed identity $w^{\prime} \in G *\l t\rangle$. Thus, it is sufficient to prove that $G$ does not satisfy any mixed identity  $w \in (G*\l t \r) \setminus \l t \r$.

Take any reduced nonempty word $w \in (G*\l t \r) \setminus \l t \r$. Up to conjugation $w = t^{n_1}a_1t^{n_2}a_2 \ldots a_{m-1}t^{n_m}a_m$, where $a_i \in G\setminus \{1\}$, $n_i \in \ZZ \setminus \{0\}$, $i \in \{1, \ldots, m\}$. 
    
Since $G$ acts on $\Lambda_S(G)$ by homeomorphisms, and the action is topologically free, the sets $\Fix_{\Lambda_S(G)}(\l a_i \r)$ are closed in $\Lambda_S(G)$ and $\Int(\Fix_{\Lambda_S(G)}(\l a_i \r)) = \emptyset$ for every $i \in \{1, \ldots, m\}$. Let 
    $$M = \bigcap\limits_{i \in \{1 \ldots m\}}\Lambda_S(G) \setminus\Fix_{\Lambda_S(G)}(\l a_i \r).$$
The set $M$ is a finite intersection of open dense sets, so it is open and dense; in particular, it is nonempty, and there exists $s \in \Lambda_S(G)$ such that $a_is \neq s$ for every $i \in \{1, \ldots, m\}$.
        
Since $\Lambda_S(G)$ is Hausdorff, there exist open subsets $V_0$ and $V_1, \ldots, V_m$ of $\Lambda_S(G)$ such that $s \in V_0$, $a_is \in V_i$ and $ V_0 \cap V_i =\emptyset$ for every $i \in \{1, \ldots, m\}$. Take $U = V_0 \cap a_1^{-1}(V_1) \cap \ldots \cap a_m^{-1}(V_m)$ to ensure that $U \cap a_i(U) = \emptyset$ for all $i \in \{1, \ldots, m\}$. 
    
The action of $G$ on $S$ is of general type and $U \setminus \{s\}$ is open; therefore, by Lemma \ref{set of fixed points of loxodromic elements is dense}, there exists $f \in \mathcal{L}_S(G)$ such that $\{f^+, f^-\} \subset  U \setminus \{s\}$. Let $F^+$ and $F^-$ be two disjoint open subsets of $ U \setminus \{s\}$ such that $f^+ \in F^+$ and $f^- \in F^-$. By Lemma \ref{north-south} there exists $n \in \NN$ such that
$$
f^n(\Lambda_S(G) \setminus F^-) \subseteq F^+ \text{ and } f^{-n}(\Lambda_S(G) \setminus F^+) \subseteq F^-.
$$
Take $g=f^n$. We have
 \begin{multline*}
    w(g)s = g^{n_1}a_1 \ldots a_{m-1}g^{n_m}a_ms \in g^{n_1}a_1 \ldots a_{m-1}g^{n_m}a_m(U) \subset g^{n_1}a_1 \ldots a_{m-1}g^{n_m}(\partial S \setminus U) \subset \\
    \subset g^{n_1}a_1 \ldots a_{m-1}(F^+ \cup F^-) \subset g^{n_1}a_1g^{n_2}a_2 \ldots a_{m-1}(U) \subset \ldots \subset g^{n_1}a_1(U) \subset F^+ \cup F^-.  
    \end{multline*}
    Since $s \not\in F^+ \cup F^-$, we obtain that $w(g)s \neq s$; in particular, $w(g) \neq 1$ in $G$.
    \end{proof}

\begin{proof}[Proof of Theorem \ref{main_theorem}]
Condition (b) obviously implies (a). The equivalence of conditions (a) and (c) is the statement of Theorem \ref{real main theorem}. Being MIF is a purely algebraic property of a group. So, by Theorem \ref{real main theorem}, if $G$ is MIF, then any general type action on any hyperbolic space $S$ should induce a topologically free action on $\Lambda_S(G)$ whenever the action on $\Lambda_S(G)$ is faithful. By the definition of an elliptic radical, the induced action of $G$ on $\Lambda_S(G)$ is faithful if and only if $E_S(G)$ is trivial. Therefore, (c) implies (b).
\end{proof}

\begin{proof}[Proof of Corollary \ref{algebraic subgroups}]
Let $H$ be a proper normal subgroup of $G$. Then, by Lemma \ref{lim-free passes to normal subgroups}, $H$ is itself lim-free, and $\Lambda_S(H)=\Lambda_S(G)$. Suppose that $H$ is algebraic. Then there exists a finite collection of mixed identities $\{ w_1(t), \dots, w_n(t)\}$ in $G * \langle t \rangle$ such that for any $h \in H$ there exists $i \in \{1, \dots, n\}$ such that $w_i(h)=1$ in $G$. Up to conjugation each $w_i(t)$ is either simply $t^m$ for some $m \in \NN$ or of the form $w_i(t) = t^{n_1}a_1t^{n_2}a_2 \ldots a_{m-1}t^{n_m}a_m$, where $a_j \in G\setminus \{1\}$, $n_j \in \ZZ \setminus \{0\}$, $j \in \{1, \ldots, m\}$. Let $\{a_1, \dots a_k\}$ be the finite collection of all coefficients from $G$ appearing in the mixed identities $\{ w_1(t), \dots, w_n(t)\}$. Repeating the argument of Theorem \ref{real main theorem}, we can find $s \in \Lambda_S(H)$, such that $a_is \neq s$ for every $i \in \{1, \ldots, k\}$, and $h \in \mathcal{L}_S(H)$ such that $w_i(h)s \neq s$ simultaneously for all $w_i(t)$ containing nontrivial coefficients from $G$. Note that $h^m \neq 1$ in $G$ for every $m \in \NN$, since $h$ is loxodromic. In particular, $w_i(h) \neq 1$ for all $i=1, \dots, n$ in $G$, which contradicts that $H$ is algebraic.
\end{proof}

Now, we show that faithful weakly hyperbolic groups that are not lim-free also have an interesting algebraic property.

\begin{proof}[Proof of Proposition \ref{monolithic groups}] 
 Let $S$ be a hyperbolic space such that the action of $G$ on $S$ is of general type, and the induced action of $G$ on $\Lambda_S(G)$ is faithful and not topologically free. Let $\mathcal{C}$ be the collection of all nonempty non-dense open subsets of $\Lambda_S(G)$. Let $H=\left\langle \rist_G(Y) \mid Y \in \mathcal{C}\right\rangle$. By Lemma \ref{Not_top_free actions are micro-supported}, the action of $G$ on $\Lambda_S(G)$ is micro-supported; in particular, $H$ is a nontrivial normal subgroup of $G$. Since the action of $G$ on $\Lambda_S(G)$ is faithful, $H$ acts on $\Lambda_S(G)$ nontrivially. Therefore, by Lemma \ref{dichotomy}, the action of $H$ on $S$ is also of general type and $\Lambda_S(H) = \Lambda_S(G)$. Thus, by Lemma \ref{Gen type actions admit compressible open sets}, the action of $H$ on $\Lambda_S(G)$ is fully compressible. In particular, $\mathcal{C}$ satisfies the conditions of \cite[Corollary 1.7]{GR}, and $\Mon(G)=[H,H]$ is non-abelian and simple. Assume $\Mon(G)$ is finite. Then $\Mon(G)$ is a normal elliptic subgroup of $G$, and, by Proposition \ref{elliptic radical is maximal}, $\Mon(G) \leq E_S(G)$. In particular, $E_S(G) \neq 1$, which contradicts that the action of $G$ on $\Lambda_S(G)$ is faithful.
\end{proof}

Proposition \ref{monolithic groups} does not apply to any acylindrically hyperbolic group. In fact, every acylindrically hyperbolic group has a trivial monolith.

\begin{prop}\label{monolith is trivial}
Let $G$ be an acylindrically hyperbolic group. Then $\Mon(G)=1$.
\end{prop}

\begin{proof}
By \cite[Thm 2.24 (c)]{DGO}, there exists a hyperbolically embedded subgroup $H \leq G$ such that $H \cong \ZZ \times K(G)$, where $K(G)$ is the unique maximal finite normal subgroup of $G$. Let $g$ be any nontrivial element of $G$. Then, by \cite[Thm 7.19]{DGO}, for the finite subset $Z=\{g,1\}$  there exists a finite subset $F \subseteq H \setminus \{1\}$ such that for every normal subgroup $N$ of $H$, satisfying $N \cap F =\emptyset$, the natural epimorphism $\varepsilon: G \rightarrow \overline{G}$ is injective on $Z$, where 
$$
\overline{G}=G /\left\langle\left\langle N\right\rangle\right\rangle^G.
$$
Since $F$ is finite and $H \cong \ZZ \times K(G)$, a nontrivial normal subgroup $N$ with the property $N \cap F =\emptyset$ always exists. Thus, $\overline{G}$ is a proper quotient, in which the image of $g$ is nontrivial; therefore, the kernel of the natural homomorphism avoids $g$. Since $g$ was an arbitrary nontrivial element of $G$,  $\Mon(G)=1$.
\end{proof}

This allows us to obtain the result of D. Osin and M. Hull  \cite[Cor. 5.10]{OH} independently, as a corollary of Proposition \ref{monolith is trivial}, Theorem \ref{main_theorem}, and Proposition \ref{monolithic groups}.

\begin{cor}\label{acylindrically hyperbolic groups are lim-free}
If the action of a group $G$ on a hyperbolic space $S$ is of general type and acylindrical, then the following conditions are equivalent
\begin{enumerate}
    \item The finite radical $K(G)$ is trivial.
    \item The action of $G$ on $\Lambda_S(G)$ is faithful.
    \item The action of $G$ on $\Lambda_S(G)$ is topologically free.
    \item $G$ is MIF.
\end{enumerate}
\end{cor}
\begin{proof}
The equivalence of $(1)$ and $(2)$ follows from $K(G)=E_S(G)$, by Lemma \ref{E(G)=K(G)}.
To obtain $(2) \Longrightarrow (3)$, note that by Proposition \ref{monolith is trivial}, $\Mon(G)=1$, and by Proposition \ref{monolithic groups}, $G$ is lim-free.
$(3) \Longleftrightarrow (4)$ is Theorem \ref{main_theorem}, and $(4) \Rightarrow (1)$ is obvious.
\end{proof}

Proposition \ref{monolithic groups} also gives us another large class of lim-free weakly hyperbolic groups.

\begin{proof}[Proof of Corollary \ref{res finite groups}]
If a faithful weakly hyperbolic residually finite group $G$ is not lim-free, then, by Proposition \ref{monolithic groups}, $\Mon(G)$ is infinite and simple. This contradicts the fact that every subgroup of a residually finite group is residually finite. By a well-known result of Malcev \cite{Mal}, every finitely generated linear group is residually finite. Thus, every finitely generated faithful weakly hyperbolic linear group is lim-free. By Theorem \ref{main_theorem}, every faithful weakly hyperbolic residually finite group is MIF. 
\end{proof}

In the special case of subgroups of  $\PSL_2(\CC)$, we do not need finite generation.

\begin{proof}[Proof of Corollary \ref{generalized Tits Alternative}]
Let $\HH^3=\{z+t j: z \in \CC, t>0\}$ be the upper half-space model of hyperbolic three-space.
$$\PSL_2(\CC)=\left\{\begin{pmatrix} a & b\\ c & d \end{pmatrix} \;   \vrule \; a, b, c, d \in \CC, \;  ad-bc=1  \right\} \big / \left\{ \pm I \right \}.$$ 
Every subgroup $H$ of $\PSL_2(\CC)$ acts on $\mathbb H^3$ by M\"obius transformations:
$$
\begin{aligned}
&g(z)=\begin{pmatrix} a & b\\ c & d \end{pmatrix}z = \frac{az+b}{cz+d}, \; z \in \CC, \\
& g(\infty)=\infty \text{ if } c=0 \text{ and } g\left(\frac{-d}{c}\right)=\infty, \; g(\infty)=\frac{a}{c} \text{ if } c \neq 0.
\end{aligned}
$$

See \cite[Chapter 4]{Bea} for more details. 
The boundary of $\HH^3$ is $\widehat{\CC}=\CC \cup\{\infty\}$. Analyzing solutions of the equation 
$cz^2+(d-a)z-b=0$, where $z \in \Lambda_{\HH^3}(H)$, we obtain that an element of $H$ has more than two fixed points in $\Lambda_{\HH^3}(H)$ if and only if $c=b=0$ and $d=a$, which happens only when it is trivial. Thus, the action of $H$ on $\Lambda_{\HH^3}(H)$ is topologically free. Therefore, by Theorem \ref{main_theorem}, if the action of $H$ on $\HH^3$ is of general type, then $H$ is MIF.

Assume that the action of $H$ on $\HH^3$ is not of general type. Then either $H$ fixes a point on $\partial \HH^3$ or the action of $H$ is elliptic or lineal, exchanging the two points in the limit set. 

If the action of $H$ is elliptic, then $H$ fixes a point in $\HH^3$ \cite[Thm 4.3.7]{Bea}. Therefore, $H$ is conjugate to a subgroup of $\SO(3)$. Every proper closed (in the Euclidean topology) subgroup of $\SO(3)$ is virtually abelian. Therefore, if $H$ is not virtually solvable, then $H$ is dense in $\SO(3)$. Therefore, $H$ is Zariski dense, and, by the result of Tomanov \cite{Tom}, is MIF.

If the action of $H$ is lineal, exchanging the two points in the limit set, then $H$ contains a subgroup $N$ of index $2$, fixing these points in $\partial \HH^3$.

Now let us consider the case in which $H$ fixes a point in $\partial \HH^3$. Since the action of $\PSL_2(\CC)$ is transitive on  $\widehat{\CC}$, we can always find an element $g \in \PSL_2(\CC)$ such that $H' = g^{-1}Hg$ fixes $\infty$. 
Then $H'$ consists of the upper-triangular matrices 
$$\left\{\begin{pmatrix} a & b\\ 0 & d \end{pmatrix} \;   \vrule \; a, b, d \in \CC, \;  ad=1  \right\}$$
and is solvable. Therefore, $H$ is also solvable.
\end{proof} 

\begin{rem}
Unfortunately, Corollary \ref{res finite groups} does not generalize to arbitrary finitely generated weakly hyperbolic linear groups; the faithfulness condition is essential here. Neither does Corollary \ref{generalized Tits Alternative} generalize to subgroups of projective linear groups of higher dimension. For example, $F_2 \times \ZZ \leq \SL_3(\RR)$ is not virtually solvable, but also is not MIF. 
\end{rem}

The next corollary shows that results about mixed identities from \cite{FMMS} can be obtained as a special case of Theorem \ref{real main theorem}.

\begin{cor}\cite[Theorem B]{FMMS}\label{previous results about MIF in trees}
Let $T$ be a simplicial tree. Suppose that the action of $G$ on $T$ is faithful, minimal, and of general type. Then the following are equivalent:
\begin{enumerate}
\item $G$ is MIF;
\item The action of $G$ on $\partial T$ is topologically free.
\end{enumerate} 
\end{cor}
\begin{proof}
To obtain the statement as a corollary of Theorem \ref{real main theorem} we need to show that if the action of $G$ on $T$ is faithful, minimal, and of general type, then the induced action on $\Lambda_T(G)$ is also faithful.
By Lemma \ref{limit set is the whole boundary} $\Lambda_T(G) = \partial T$. By taking the barycentric subdivision, we can assume that the action of $G$ on $T$ is without inversions. Therefore, by Lemma \ref{fixed point in a tree}, $E_T(G)$ fixes a vertex $t \in T$. Since $E_T(G)$ is normal in $G$, it fixes every vertex of the orbit $Gt$. Thus, it stabilizes a minimal invariant subtree $T_{Gt}$ containing this orbit. Since the action of $G$ on $T$ is minimal, $T_{Gt}=T$, and, since the action of $G$ on $T$ is faithful, $E_T(G)=1$. Therefore, the induced action of $G$ on $\partial T$ is also faithful. 
\end{proof}

We finish this section by proving Corollary \ref{selfless groups}, connecting our research to reduced $C^*$-algebras. This follows immediately from the following theorem, which is the combination of Proposition 15 and Theorem 14 from \cite{Oza}.

 \begin{thm}\label{Ozawa}
 Let $G$ be a group. Assume that there is a continuous action of $G$ on a Hausdorff topological space that is minimal, fully compressible, and topologically free.
 Then the reduced group $C^*$-algebra $C_\lambda^*(G)$ is selfless. 
 \end{thm}

\begin{proof}[Proof of Corollary \ref{selfless groups}]
Suppose a group $G$ admits a general type action on a hyperbolic space $S$ such that the induced action on $\Lambda_S(G)$ is topologically free. Then, by Lemma \ref{Gen type actions admit compressible open sets} and Corollary \ref{action on the limit set is minimal}, the action of $G$ on $\Lambda_S(G)$ is fully compressible and minimal. Therefore, by Theorem \ref{Ozawa}, the reduced group $C^*$-algebra $C_\lambda^*(G)$ is selfless.  
\end{proof}

\subsection{Boundary dynamics and triple transitivity }\label{conjugate orbits}
The goal of this subsection is to prove Theorem \ref{thm about conjugate orbits}. The proof strategy closely follows \cite{BM}, but the details are different. The key ingredient of the proof is the following proposition.

\begin{prop}\label{non-elementary action}
 Suppose that the action of a group $G$ on a hyperbolic space $S$ is of general type. Let $\Omega$ be a set on which $G$ acts faithfully and 2-transitively. Then for any $\omega \in \Omega$, the action of a subgroup $G_\omega$ on $S$ is non-elementary, and $\Lambda_S(G_\omega)=\Lambda_S(G)$.
\end{prop}

 We will need the following known and easily established fact (see, for example, \cite[ex. 47]{La}). We provide the proof to make the paper self-contained. 
     \begin{lem}\label{double cosets}
        If the action of a group $G$ on a set $\Omega$ is $2$-transitive, then for any $\omega \in \Omega$, there exists an element $g \in G$ such that $G$ admits a decomposition into double cosets $G=G_\omega \sqcup G_\omega gG_\omega$.
    \end{lem} 
    \begin{proof}
       Fix $\omega \in \Omega$. Take any $\omega_1, \omega_2 \in \Omega\setminus\{\omega\}$.
       Since the action of $G$ on $\Omega$ is $2$-transitive, the action of $G_\omega$ is transitive on $\Omega \setminus \{\omega\}$. Therefore, there exists $\gamma \in G_\omega$ such that $\gamma \omega_1=\omega_2$. Since $G$ acts transitively on $\Omega$, there exists $g \notin G_\omega$ such that $g\omega = \omega_1$. Hence, $\omega_2=\gamma g\omega$. Since $\omega_2$ was arbitrary, we obtain $\Omega\setminus\{\omega\} = G_\omega g\omega$. Thus, if $f \in G\setminus G_\omega$, then $f\omega=\gamma_1g\omega$ for some $\gamma_1 \in G_\omega$. So, $f=\gamma_1g\gamma_2$ for some $\gamma_2 \in G_\omega$. Therefore, $G= G_\omega \sqcup G_\omega gG_\omega$.
    \end{proof}

\begin{rem}\label{maximality of H}
    Note that the subgroup $G_\omega \leq G$ in this case is maximal.
\end{rem}

\begin{proof}[Proof of Proposition \ref{non-elementary action}]

Fix an arbitrary $\omega \in \Omega$. The action of $G$ on $\Omega$ is $2$-transitive; thus, by Lemma \ref{double cosets}, there exists $g \in G$ such that $G=G_\omega \sqcup G_\omega gG_\omega$. 
We argue by contradiction and assume that $\Lambda_S(G_\omega) \neq \Lambda_S(G)$. 

{\bf Case 1.} Assume that $\Lambda_S(G_\omega)=\emptyset$, so the action of $G_\omega$ on $S$ is elliptic. Take any point $s\in S$. Let $K=\sup\{\d_S(ts,s) \mid t \in G_\omega\}$, and $M=\d_S(gs,s)$. Let $f \in G$ be arbitrary. If $f \in G_\omega$, then $\d_S(fs, s) \leq K$. If $f \notin G_\omega$, then, by Lemma \ref{double cosets}, there exist $\gamma_1, \gamma_2 \in G_\omega$, such that $f=\gamma_1g\gamma_2$. Thus, we have 
$$
\d_S(fs,s)=\d_S(\gamma_1g\gamma_2s,s)\leq \d_S(g\gamma_2s,gs)+\d_S(gs,s)+\d_S(s,\gamma_1^{-1}s) \leq 2K+M. 
$$
Therefore, $G$ also acts on $S$ with bounded orbits, which contradicts that the action of $G$ on $S$ is of general type.

{\bf Case 2.} 
 Assume that $|\Lambda_S(G_\omega)|=1$, so the action of $G_\omega$ on $S$ is parabolic. Let $q(\omega)=\Lambda_S(G_\omega)$. Take $l \in \mathcal{L}_S(G)$ with sufficiently large asymptotic translation length $\tau_S(l)$ and fixed points $\{l^+, l^-\}$ disjoint from $q(\omega)$. 
Let $L_l$ be the standard $(2,\delta)$-quasi-geodesic axis of $l$, provided by Lemma \ref{uniform axis}. Let $T$ be the translation distance of $l$ along $L_l$. Let $N$ be a constant from Lemma \ref{Distance formula for loxodromic elements}, and let $m, n \geq N$ be such that $m \gg 4n$. Since $G_\omega$ contains no loxodromic elements, by Lemma \ref{double cosets}, $l^m$ and $l^n$ are in the same double coset of $G_\omega$, and there exist $\gamma_1, \gamma_2 \in G_\omega$ such that $l^m=\gamma_1l^n\gamma_2$.

By Proposition \ref{elements of parabolic subgroups don't move some points far}, there exists $s \in S$ such that 
$$
\d_S(s, \gamma_1s) \leq 56\delta \text{ and } \d_S(s, \gamma_2s) \leq 56\delta.
$$
We obtain the following: 
\begin{equation}\label{estimating distance for ln by lm}
\begin{aligned}
    &\d_S(l^ms,s)=\d_S(\gamma_1l^n\gamma_2s,s)\leq \d_S(s, \gamma_1s)+\d_S(\gamma_1s,\gamma_1l^n\gamma_2s)\leq\\
    &\leq \d_S(s, \gamma_1s)+\d_S(s,l^ns)+\d_S(l^ns, l^n\gamma_2s) \leq \d_S(s,l^ns) + 112\delta.
\end{aligned}
\end{equation}
By Lemma \ref{Distance formula for loxodromic elements}, there exists $D(\delta) \in \NN$ such that 
\begin{equation}\label{applying formula}
2\d_S(s, L_l)+\frac{m}{2}T - D \leq \d_S(s, l^ms) \text{ and } \d_S(s, l^ns)\leq 2\d_S(s, L_l)+2nT + D.
\end{equation}
Combining (\ref{estimating distance for ln by lm}) and (\ref{applying formula}), we get
$$
mT \leq 4nT+4(D+56\delta),
$$
which contradicts that $m \gg 4n$ and $T>0$.

{\bf Case 3.} Now, assume that $|\Lambda_S(G_\omega)| > 1$. 

For an element $f \in G$ and a point $o \in S$, define
$$
[\Lambda_S(G_\omega), f\Lambda_S(G_\omega)]_o:= \inf \{ [a,b,c,d]_o \mid a,b \in \Lambda_S(G_\omega), \;  c,d \in f\Lambda_S(G_\omega) \},
$$
if such points can be chosen pairwise distinct, so that $[a,b,c,d]_o$ is well-defined; and $[\Lambda_S(G_\omega), f\Lambda_S(G_\omega)]_o:=0$, otherwise.


 Let $f \in G$ be arbitrary. By Lemma \ref{double cosets}, there exists $\gamma_1, \gamma_2 \in G_\omega$ such that $f=\gamma_1g\gamma_2$. Therefore, $$
 [\Lambda_S(G_\omega),f\Lambda_S(G_\omega)]_o= [\Lambda_S(G_\omega),\gamma_1g\gamma_2\Lambda_S(G_\omega)]_o=[\gamma_1^{-1}\Lambda_S(G_\omega), g\gamma_2\Lambda_S(G_\omega)]_{\gamma_1^{-1}o}.
$$

Thus, since $\gamma_1$ and $\gamma_2$ preserve $\Lambda_S(G_\omega)$ setwise, using Lemma \ref{cross-ratio doesn't depend on o}, we obtain
$$
\begin{aligned}
    &|[\Lambda_S(G_\omega),f\Lambda_S(G_\omega)]_o-[\Lambda_S(G_\omega),g\Lambda_S(G_\omega)]_o|\\
    &\leq|[\gamma_1^{-1}\Lambda_S(G_\omega), g\gamma_2\Lambda_S(G_\omega)]_o-[\Lambda_S(G_\omega),g\Lambda_S(G_\omega)]_o|+48 \delta \leq 48\delta.
\end{aligned}
$$ 
In particular, we have
\begin{equation}\label{distance is bounded}
\sup\{[\Lambda_S(G_\omega),f\Lambda_S(G_\omega)]_o \mid f \in G \} < \infty.
\end{equation}

The subset $U=\Lambda_S(G) \setminus \Lambda_S(G_\omega)$ is nonempty and open in $\Lambda_S(G)$. Therefore, by Lemma \ref{set of fixed points of loxodromic elements is dense}, there exists $l \in \mathcal{L}_S(G)$ such that $\{l^+, l^- \} \subset U$. There exists $N \in \NN$ such that $V_N=\mathcal U_o(l^+, N) \cap \Lambda_S(G) =\{ x \in \Lambda_S(G) \mid  \l x| l^+ \r_o > N \} \subset U$. There exists $M \gg N$ such
that $V_M = \mathcal{U}_o(l^+,M) \cap \Lambda_S(G) \subset V_N \subset U$.
By Lemma \ref{north-south} there exists $m \in \NN$ such that $l^m\Lambda_S(G_\omega) \subset V_M$. Take $f=l^m$. Then, using Lemma \ref{lemma 2.2.2}, for arbitrary $c, d \in f\Lambda_S(G_\omega)$, we obtain
\begin{equation}\label{dot product c d is large}
\l c|d \r_o \geq \min\{\l c|l^+ \r_o, \l l^+|d \r_o \} -\delta \geq M -\delta.
\end{equation}
Also, for arbitrary $a, b \in \Lambda_S(G_\omega)$, we have
$$
\begin{aligned}
& N > \l a | l^+ \r_o \geq \min\{ \l a | d \r_o, \l d| l^+ \r_o \} - \delta, \text{ and } N > \l b | l^+ \r_o \geq \min\{ \l b | c\r_o, \l c| l^+ \r_o \} - \delta.
\end{aligned}
$$
Note that $\min\{ \l a | d \r_o, \l d| l^+\r_o \}=\l a | d \r_o$, and $\min\{ \l b| c \r_o, \l c| l^+ \r_o \}=\l b| c \r_o$ since $M \gg N$.
Thus, we obtain
\begin{equation}\label{dot products ad and bc are bounded}
\l a | d \r_o < N+\delta \text{ and } \l b| c \r_o < N+\delta.    
\end{equation}

Combining (\ref{dot product c d is large}) and (\ref{dot products ad and bc are bounded}), we get $[\Lambda_S(G_\omega), f\Lambda_S(G_\omega)]_o > M-2N-3\delta$, which contradicts (\ref{distance is bounded}), since $M$ was arbitrarily large. Therefore, $\Lambda_S(G_\omega)=\Lambda_S(G)$. The action of $G_\omega$ on $S$ is non-elementary, since $|\Lambda_S(G_\omega)|=|\Lambda_S(G)|=\infty$. 
\end{proof}

From now on, we fix an action of a group $G$ on a hyperbolic space $S$ that is of general type, and the induced action on $\Lambda_S(G)$ is faithful and not topologically free. Actions of subgroups of $G$ on $S$ are induced by this action. We assume that $G$ acts faithfully and $3$-transitively on a set $\Omega$. Note that $G$ is infinite, and thus $\Omega$ is also infinite, since the action of $G$ on $\Omega$ is faithful.

We want to show that for any $\omega \in \Omega$, the action of $G_\omega$ on $S$ is quasi-parabolic. Note that by the transitivity of $G$ on $\Omega$, all subgroups $G_\omega$, $\omega \in \Omega$ are conjugate, and thus their actions on $S$ are of the same type. Proposition \ref{non-elementary action} shows that it is either quasi-parabolic or of general type. Arguing by contradiction, we assume that the action of $G_\omega$ is of general type. The action of $G_\omega$ on $\Omega \setminus \{\omega\}$ is $2$-transitive. Therefore, we can apply Proposition \ref{non-elementary action} to point stabilizers of points of $\Omega \setminus \{\omega\}$ in $G_\omega$ to conclude that they also induce a non-elementary action on $S$. This gives us two possibilities to exclude: the action of such point stabilizers on $S$ is of general type, or it is quasi-parabolic. The former possibility will be excluded in Proposition \ref{Stab Delta is not of general type}, the latter in Proposition \ref{Stab Delta is not quasi-parabolic}. 

Denote by $\Omega^{(2)}$ the set of ordered pairs of elements of $\Omega$ minus the diagonal, i.e.,
$$
\Omega^{(2)}= \Omega \times \Omega \setminus  \left \{ (\omega, \omega) \mid \omega \in \Omega \right \}.
$$

\begin{prop}\label{Stab Delta is not of general type}
    There does not exist a set $\Omega$ such that the action of $G$ on $\Omega$ is faithful, $3$-transitive, the action of $G_\omega$ on $S$ is of general type for every $\omega \in \Omega$, and for every $\Delta \in \Omega^{(2)}$, the action of $G_\Delta$ on $S$ is also of general type. 
\end{prop}

\begin{proof}
    For the sake of contradiction, assume that the action of $G_\Delta$ on $S$ is of general type for some $\Delta \in \Omega^{(2)}$. Since the action of $G$ on $\Omega^{(2)}$ is transitive, all subgroups $G_\Delta$ are conjugate. Therefore, it is sufficient to arrive at a contradiction for a specifically chosen $\Delta \in \Omega^{(2)}$. We will explain the choice of $\Delta$ below.   
   
    By Proposition \ref{not 4-transitivity}, the action of $G$ on $\Lambda_S(G)$ is not topologically $4$-transitive, therefore there exist nonempty open subsets $V_1, V_2, V_3, V_4 \subset \Lambda_S(G)$ and $W_1, W_2, W_3, W_4 \subset \Lambda_S(G)$ such that there does not exist any $f \in G_\Delta$  such that
    $$
    W_i \cap f(V_i) \neq \emptyset \text{ for all } i=1,2,3,4.
    $$
    Thus, setting $V_j=W_{j-4}$ for $j=5,6,7,8,$ we obtain the list of nonempty open subsets $V_1, \dots, V_8$ with the property that for every $f \in G_\Delta$ there is a pair $i,j \in \{1, \dots, 8\}$ such that 
    \begin{equation}\label{something is disjoint}
V_i \cap fV_j = \emptyset. 
    \end{equation}
By taking smaller subsets if necessary, we can assume $V_i$ to be pairwise disjoint and not dense in $\Lambda_S(G)$ for all $i \in \{1,\dots, 8\}$. 
  
  Since the action of $G$ on $\Lambda_S(G)$ is faithful and not topologically free, by Lemma \ref{Not_top_free actions are micro-supported}, the action of $G$ on $\Lambda_S(G)$ is micro-supported. Therefore, there exists a nontrivial element $g_1 \in G$ such that $\supp(g_1) \subseteq V_1$. By Lemma \ref{non-abelian rigid stabilizers}, $\rist_G(V_1)$ cannot have all its elements of order two. Thus, we may assume that the order of $g_1$ is greater than two, and, since the action of $G$ on $\Omega$ is faithful, $g_1$ has an orbit of size at least $3$ in $\Omega$. Let $\omega_1, \omega_2, \omega_3 \in \Omega$ be distinct points such that
  \begin{equation}\label{choice of delta}
  \omega_2=g_1\omega_1, \text{ and } \omega_3=g_1\omega_2.
  \end{equation}
  Define $\Delta= (\omega_1, \omega_2)$. 

Since the action of $G_\Delta$ on $S$ is of general type, by Lemma \ref{Gen type actions admit compressible open sets}, we can find $u_2, u_3, u_4, u_5, u_6, u_7, u_8 \in G_\Delta$ such that $g_i=u_ig_1u_i^{-1}$ have supports in $V_i$, respectively, $i=2, \dots, 8$. 

Using (\ref{choice of delta}), the fact that each $u_i$ stabilizes both $\omega_1$ and $\omega_2$, and that $[g_i, g_j]=1$ for $i,j \in \{1, \dots, 8\}, \, i\neq j$, since they have disjoint support, for all $i=2,\dots, 8$ we obtain
\begin{equation}\label{1 and 2}
    \begin{aligned}
    &g_i\omega_1=u_ig_1u_i^{-1}\omega_1=u_ig_1\omega_1=u_i\omega_2=\omega_2, \text{ and}\\
    &g_i\omega_2=g_ig_1\omega_1=g_1g_i\omega_1=g_1\omega_2=\omega_3.
    \end{aligned}
\end{equation}

Take any $f \in G_\Delta$. 
By (\ref{something is disjoint}), elements $g_i^f$ and $g_j$ have disjoint support for at least one pair $i,j \in \{1,\dots, 8\}$ and, thus, $
[g_i^f,g_j]=1$.
Using this, (\ref{1 and 2}), and the fact that $f$ fixes $\omega_1$ and $\omega_2$, we obtain
\begin{equation}\label{w3}
\begin{aligned}
&g_i^f\omega_1=f^{-1}g_i\omega_1=f^{-1}\omega_2=\omega_2 \text{ and }\\
&g_i^f\omega_2 =g_i^fg_j\omega_1=g_jg_i^f\omega_1=g_j\omega_2 =\omega_3.
\end{aligned}
\end{equation}
Combining (\ref{w3}) and $g_i^f\omega_2=f^{-1}g_i\omega_2=f^{-1}\omega_3$, we obtain
$ f\omega_3=\omega_3$, which contradicts the assumption that the action of $G$ on $\Omega$ is $3$-transitive.
\end{proof}

\begin{prop}\label{Stab Delta is not quasi-parabolic}
    There does not exist a set $\Omega$ such that the action of $G$ on $\Omega$ is faithful, $3$-transitive, the action of $G_\omega$ on $S$ is of general type for every $\omega \in \Omega$, and for every $\Delta \in \Omega^{(2)}$, the action of $G_\Delta$ on $S$ is quasi-parabolic. 
\end{prop}

   Denote by $\Omega^{\{2\}}$ the set of unordered pairs of distinct elements of $\Omega$ and by $\Delta^{\{2\}}$ the element of $\Omega^{\{2\}}$, corresponding to $\Delta$. 
   Denote by  $G_{\Delta^{\{2\}}}$  the setwise stabilizer of $\Delta^{\{2\}}$ in $G$. Clearly, $G_\Delta$ is a subgroup of index $2$ in $G_{\Delta^{\{2\}}}$, hence is normal. For the sake of contradiction, assume that the action of $G_\Delta$ on $S$ is quasi-parabolic. Then, by the normality of $G_\Delta$ in $G_{\Delta^{\{2\}}}$, the action of $G_{\Delta^{\{2\}}}$ on $S$ is also quasi-parabolic. Note that, by Proposition \ref{non-elementary action}, $\Lambda_S(G)=\Lambda_S(G_\omega)$. We denote by $q(\Delta)$ the unique point of $\Lambda_S(G)$, which is fixed by $G_\Delta$ and $G_{\Delta^{\{2\}}}$. Note that $G_{\Delta}$ acts transitively on  $\Omega\setminus \Delta$ by the $3$-transitivity of the $G$-action on $\Omega$, and $G_{\Delta} < G_{q(\Delta)}$, so $G_{q(\Delta)}$ acts transitively on $\Omega \setminus  \Delta $.
  
   The proof of Proposition \ref{Stab Delta is not quasi-parabolic} will consist of a series of lemmas that will lead to a contradiction.

\begin{lem}\label{orbits of rigid stabilizers of size 3} 
Suppose that $G$ and $G_\omega$ satisfy the conditions of Proposition \ref{Stab Delta is not quasi-parabolic} for any $\omega \in \Omega$. Then for every nonempty open subset $U \subset \Lambda_S(G)$, all orbits of $\rist_G(U)$ in $\Omega$ have cardinality at least $3$.
\end{lem}
\begin{proof}
    Assume for contradiction that $\rist_G(U)$ has an orbit of cardinality at most $2$. Then the induced permutation group on this orbit is abelian, and the commutator subgroup $[\rist_G(U), \rist_G(U)]$ must fix a point $\omega \in \Omega$. Since the action of $G_\omega$ on $S$ is of general type, by Lemma \ref{Gen type actions admit compressible open sets}, for any nonempty non-dense open $V \subset \Lambda_S(G)$ there exists $f \in G_\omega$ such that $f(V) \subset U$. Thus, we have 
    $$f \rist_G(V)f^{-1} \leq \rist_G(U),$$ and $$f[\rist_G(V),\rist_G(V)]f^{-1} \leq [\rist_G(U), \rist_G(U)] \leq G_\omega.$$
    Since $f \in G_\omega$, we obtain $[\rist_G(V),\rist_G(V)] \leq G_\omega$.

    By Lemma \ref{Not_top_free actions are micro-supported}, the action of $G$ on $\Lambda_S(G)$ is micro-supported, and, by Lemma \ref{non-abelian rigid stabilizers}, rigid stabilizers of nonempty open subsets of $\Lambda_S(G)$ are non-abelian. Therefore, 
    $$
    N = \l [\rist_G(V),\rist_G(V)] \mid  V \subset \Lambda_S(G) \text{ is a nonempty non-dense open subset} \r
    $$ 
   is a non-trivial normal subgroup of $G$. Thus, $G_\omega$ contains a non-trivial normal subgroup of $G$. Since the action of $G$ on $\Omega$ is $2$-transitive, this means that the action of $G_\omega$ is transitive on $\Omega$, which contradicts that  $G_\omega$ fixes $\omega \in \Omega$.  
\end{proof}

\begin{lem}\label{the action of Stab is transitive on Omega}
  $G_{\Delta^{\{2\}}} \neq G_{q(\Delta)}$ for any $\Delta^{\{2\}} \in \Omega^{\{2\}}$.
\end{lem}
\begin{proof}
    Fix any $\Delta^{\{2\}}=\{\omega', \omega'' \} \in \Omega^{\{2\}}$. The orbit  $G_{\Delta^{\{2\}}}\omega'$ in $\Omega$ is of size at most $2$. For any nonempty open subset $U \subset \Lambda_S(G)$, not containing $q(\Delta)$, the subgroup $\rist_G(U)$ fixes $q(\Delta)$, and hence $\rist_G(U) \leq G_{q(\Delta)}$. By Lemma \ref{orbits of rigid stabilizers of size 3}, any orbit of $\rist_G(U)$ in $\Omega$ is of size at least $3$. Therefore, $G_{\Delta^{\{2\}}} \neq G_{q(\Delta)}$.
\end{proof}

 Define $\pi \colon \Omega^{\{2\}} \longrightarrow \Lambda_S(G)$ by $\Delta^{\{2\}} \mapsto q(\Delta)$. This map is $G$-equivariant.

\begin{lem}\label{2-transitive action preserves blocks}
    The fiber $\pi^{-1}(q(\Delta))$ forms a partition $\mathcal{P}_\Delta$ of $\Omega$ into blocks of size $2$, and $G_{q(\Delta)}$ preserves $\mathcal{P}_\Delta$ and acts $2$-transitively on its blocks.
\end{lem}
\begin{proof} 
    Fix $\Delta \in \Omega^{\{2\}}$. We first argue that two distinct $\Delta_1^{\{2\}}, \Delta_2^{\{2\}} \in \pi^{-1}(q(\Delta))$ are always disjoint. Assume that this is not the case. Then $\Delta_1^{\{2\}} \cap \Delta_2^{\{2\}}$ is a singleton, say $\Delta_1^{\{2\}} \cap \Delta_2^{\{2\}}=\{\omega\}$.  Let $\Delta_1^{\{2\}}=\{\omega, \omega_1\}$ and $\Delta_2^{\{2\}}= \{\omega, \omega_2\}$. Since the action of $G_\omega$ on $\Omega \setminus \{\omega\}$ is $2$-transitive, the actions of $G_{\Delta_1}$ and $G_{\Delta_2}$ on $\Omega \setminus \{\omega, \omega_1\}$ and $\Omega \setminus \{\omega, \omega_2\}$ are transitive, and we claim that 
    $$
    G_\omega = \langle G_{\Delta_1}, G_{\Delta_2}\rangle.
    $$
 Indeed, it is obvious that $\langle G_{\Delta_1}, G_{\Delta_2}\rangle \leq G_\omega$, so we need to show that 
    $G_\omega \leq \langle G_{\Delta_1}, G_{\Delta_2}\rangle$.
    Take an arbitrary $h \in G_\omega$. If $h\omega_i=\omega_i$, for some $i=1,2$, then $h \in G_{\Delta_i}$. Suppose $h\omega_i=\widehat{\omega}_i \neq \omega_i$ for both $i=1,2$. Without loss of generality, assume $h\omega_1= \widehat{\omega} \neq \omega_1$. Then there exists $h_1 \in G_{\Delta_1}$ such that $h_1\widehat{\omega}=\omega'$, where $\omega' \notin \{\omega_1, \omega_2\}$. Thus, we can find $h_2 \in G_{\Delta_2}$, such that $h_2\omega' =\omega_1$. Therefore, $h_2h_1h\omega_1=\omega_1$, and $h \in \langle G_{\Delta_1}, G_{\Delta_2}\rangle$. Since $\Delta_1^{\{2\}}, \Delta_2^{\{2\}} \in \pi^{-1}(q(\Delta))$, subgroups $G_{\Delta_1}$ and $G_{\Delta_2}$ both fix $q(\Delta)$. Thus, it follows that the subgroup $G_\omega$ also fixes $q(\Delta)$. This contradicts the assumption that the action of $G_\omega$ on $\Lambda_S(G)$ is of general type.

    It remains to show that the pairs of $\pi^{-1}(q(\Delta))$ cover $\Omega$. Let $\Sigma$ be the union of these pairs. Note that $\Sigma$ defines a subset of $\Omega$, and that this subset is $G_{q(\Delta)}$-invariant by the equivariance of the map $\pi$. Since $G_{q(\Delta)}$ is transitive on $\Omega$, we must have $\Sigma=\Omega$.

    The subgroup $G_{q(\Delta)}$ fixes $q(\Delta)$, thus it preserves $\mathcal{P}_{\Delta}$. The action of $G_{q(\Delta)}$ on the blocks of $\mathcal{P}_{\Delta}$ is transitive because the action of $G_{q(\Delta)}$ on $\Omega$ already is. In order to show that this action of $G_{q(\Delta)}$ on the blocks of $\mathcal{P}_{\Delta}$ is $2$-transitive, we need to show that the block stabilizers act transitively, i.e., we need to show that  $G_{\Delta^{\{2\}}}$ acts transitively on $\mathcal{P}_{\Delta} \setminus \{\Delta\}$. Since $G_\Delta < G_{\Delta^{\{2\}}}$, it is enough to show that the action of $G_\Delta$ is transitive on $\mathcal{P}_{\Delta} \setminus \{\Delta\}$. Indeed, take any $\Delta^{\prime}, \Delta^{\prime \prime} \in \mathcal{P}_{\Delta}$ that are distinct from $\Delta$. Choose $\omega^{\prime} \in \Delta^{\prime}$ and $\omega^{\prime \prime} \in \Delta^{\prime \prime}$. Since $\Delta^{\prime}, \Delta^{\prime \prime}$ are disjoint from $\Delta$ we have $\omega^{\prime}, \omega^{\prime \prime} \in \Omega \setminus \{\Delta\}$. Since $G_{\Delta}$ acts transitively on $\Omega \setminus \{\Delta\}$, there exists $f \in G_{\Delta}$ such that $f\omega^{\prime}=\omega^{\prime \prime}$. Since $G_{\Delta}$ preserves the blocks, it follows that we must have $f\Delta^{\prime}=\Delta^{\prime \prime}$, and hence $G_{\Delta}$ acts transitively on $\mathcal{P}_{\Delta} \setminus \{\Delta\}$.
\end{proof}

 By $G_{q(\Delta)}^0$ we denote the subgroup of $G_{q(\Delta)}$ consisting of elements that fix some open neighborhood $U \subset \Lambda_S(G)$ of $q(\Delta)$ pointwise, i.e.,
 $$
 G_{q(\Delta)}^0 = \bigcup \{\rist_G(V) \mid V \text{ is open and } q(\Delta) \not\in  \overline{V}\}.
 $$ 
 
 Observe that $G_{q(\Delta)}^0$ is a non-trivial normal subgroup of $G_{q(\Delta)}$ since the action of $G$ on $\Lambda_S(G)$ is micro-supported by Lemma \ref{Not_top_free actions are micro-supported}.
\begin{lem}\label{product coinsides with the stabilizer}
    For every $\Delta^{\{2\}} \in \Omega^{\{2\}}$, we have $G_{q(\Delta)}^0G_\Delta=G_{q(\Delta)}$.
\end{lem}
\begin{proof}
   Recall that a group action is called {\it primitive} if it does not preserve any non-trivial partition. By Lemma \ref{2-transitive action preserves blocks}, the action of $G_{q(\Delta)}$ on the blocks of the partition $\mathcal{P}_{\Delta}$ is $2$-transitive, thus it is primitive. Since $G_{q(\Delta)}^0$ is a normal subgroup of $G_{q(\Delta)}$, and the action of $G_{q(\Delta)}$ on the blocks of the partition $\mathcal{P}_{\Delta}$ is transitive and primitive, it follows that the action of $G_{q(\Delta)}^0$ on these blocks is either trivial or transitive. Otherwise, the non-trivial partition into $G_{q(\Delta)}^0$-orbits would be preserved by the action of $G_{q(\Delta)}$.
    
   If the action of $G_{q(\Delta)}^0$ on the blocks of $\mathcal{P}_{\Delta}$ were trivial, then $G_{q(\Delta)}^0$ would consist of elements of order at most two, since all blocks have size $2$, the action of $G_{q(\Delta)}^0$ on $\Omega$ is faithful, and $G_{q(\Delta)}^0$ preserves the partition $\mathcal{P}_{\Delta}$ as a subgroup of $G_{q(\Delta)}$. This cannot happen since $G_{q(\Delta)}^0$ contains a rigid stabilizer of some nonempty open subset $V \subset \Lambda_S(G)$, disjoint from $q(\Delta)$, and, by Lemma \ref{non-abelian rigid stabilizers}, $\rist_G(V)$ cannot be abelian. 
   
   Therefore, this action is transitive, and it follows that $G_{q(\Delta)}^0$ has at most two orbits in $\Omega$, and each orbit intersects each block of $\mathcal{P}_{\Delta}$.
   
   Consider the subgroup $G_{q(\Delta)}^0G_\Delta \leq G_{q(\Delta)}$.
   It is indeed a subgroup since $G_{\Delta}$ normalizes $G_{q(\Delta)}^0$. We claim that $G_{q(\Delta)}^0G_\Delta$ acts transitively on $\Omega$. There are two cases to consider. First, if $G_{q(\Delta)}^0$ has only one orbit, then this is clear. Second, if $G_{q(\Delta)}^0$ has two orbits, then the subgroup $G_\Delta$ does not preserve the partition of $\Omega$ into these two orbits because $G_\Delta$ acts transitively on $\Omega \setminus \{\Delta\}$. Thus, the claim follows. 

    Take $\omega \in \Delta$ and any $f \in G_\omega \cap G_{q(\Delta)}$. By the equivariance of $\pi$, we have
    $$
    q(f\Delta)=fq(\Delta)=q(\Delta),
    $$ thus $f\Delta$ and $\Delta$  belong to the same fiber of the partition. On the other hand,
    $\omega \in f\Delta \cap \Delta$ since $f$ fixes $\omega$. Therefore, $f\Delta=\Delta$ and 
    \begin{equation}\label{contains stabilizer}
    G_\omega \cap G_{q(\Delta)}=G_\Delta.
    \end{equation}
    Recall that if a group $G$ acts on a set $\Omega$ transitively and a subgroup $H \leq G$ also acts on $\Omega$ transitively, then for any $\omega \in \Omega$,
    $$
    G=H G_\omega.
    $$
    Therefore, since $G_{q(\Delta)}^0G_\Delta$ contains a stabilizer of a point in $\Omega$ under the action of $G_{q(\Delta)}$ by (\ref{contains stabilizer}), and both actions of $G_{q(\Delta)}$ and $G_{q(\Delta)}^0 G_{\Delta}$ are transitive on $\Omega$, we finally obtain
    $$
    G_{q(\Delta)}^0 G_\Delta=G_{q(\Delta)}. \qedhere
    $$
\end{proof}

\begin{lem}\label{Delta and Delta'}
    Fix $\Delta^{\{2\}} \in \Omega^{\{2\}}$ and $\Delta'^{\{2\}} \in \mathcal{P}_\Delta$ distinct from $\Delta^{\{2\}}$. Then there exists a nonempty open subset $V \subset \Lambda_S(G)$ such that $q(\Delta) \notin \overline{V}$ and $h \in \rist_G(V)$ such that $h\Delta^{\{2\}} = \Delta'^{\{2\}}$ and $h\Delta'^{\{2\}} \neq \Delta^{\{2\}}$.
\end{lem}

\begin{proof}
Consider the action of $G_{q(\Delta)}^0$ on $\Omega$, and denote by 
$$
\psi \colon G_{q(\Delta)}^0 \longrightarrow \Sym(\mathcal{P}_{\Delta})
$$
the induced action on the set of blocks of $\mathcal{P}_{\Delta}$.
We have a short exact sequence 
$$
1 \longrightarrow \ker(\psi) \longrightarrow G_{q(\Delta)}^0 \longrightarrow \psi(G_{q(\Delta)}^0) \longrightarrow 1.
$$
Note that $\ker(\psi)$ is an elementary abelian 2-group, since $G_{q(\Delta)}^0$ preserves the partition $\mathcal{P}_{\Delta}$, which consists of blocks of size $2$.

 The subgroup $G_{q(\Delta)}^0$ contains a rigid stabilizer of some nonempty open subset $V \subset \Lambda_S(G)$ with $q(\Delta) \notin \overline{V}$, thus, by Lemma \ref{wreath products in rigid stabilizers}, it does not consist of elements of order not greater than $4$. Therefore, it is not an extension of two elementary abelian 2-groups, and  $\psi(G_{q(\Delta)}^0)$ does not consist only of elements of order $2$. Thus, there exists an element of $G_{q(\Delta)}^0$ having a cycle of length at least $3$ under the action of $G_{q(\Delta)}^0$ on $\mathcal{P}_{\Delta}$, i.e., there exist an element $h_0$ and distinct $\Delta_1^{\{2\}}, \Delta_2^{\{2\}}, \Delta_3^{\{2\}} \in \mathcal{P}_{\Delta}$ such that $$
h_0\Delta_i^{\{2\}}=\Delta_{i+1}^{\{2\}}, i=1,2.
$$ 

By the definition of $G_{q(\Delta)}^0$, there exists an open subset $V_0 \subset \Lambda_S(G)$, not containing $q(\Delta)$, such that $h_0 \in \rist_G(V_0)$. 

Now, since the action of $G_{q(\Delta)}$ on $\mathcal{P}_{\Delta}$ is $2$-transitive by Lemma \ref{2-transitive action preserves blocks}, we may find $f \in G_{q(\Delta)}$ such that $f\Delta_1^{\{2\}}=\Delta^{\{2\}}$ and $f\Delta_2^{\{2\}}=\Delta'^{\{2\}}$.
Let $V=fV_0$. Note that $ q(\Delta) \notin V$, since $f$ fixes $q(\Delta)$.
Consider the element $h=f h_0 f^{-1}$. We have $h \in \rist_G(V)$, and 
$$
\begin{aligned}
&h\Delta^{\{2\}}= fh_0f^{-1}\Delta^{\{2\}}=fh_0\Delta_1^{\{2\}}=f\Delta_2^{\{2\}}=\Delta'^{\{2\}};\\
&h\Delta'^{\{2\}}= fh_0f^{-1}\Delta'^{\{2\}}=fh_0\Delta_2^{\{2\}}=f\Delta_3^{\{2\}} \neq \Delta^{\{2\}}.
\end{aligned}
$$
Therefore, $h$ satisfies the conclusion.
\end{proof}

\begin{lem}\label{L:3.13}
    Fix $\Delta^{\{2\}} \in \Omega^{\{2\}}$. Then for every nonempty open subset $U \subset \Lambda_S(G)$ there exists $g \in \rist_G(U)$ such that $g$ exchanges the two elements of $\Delta^{\{2\}}$.
\end{lem}
\begin{proof}
Choose an element $f \in G$ such that $f$ exchanges the two elements of $\Delta^{\{2\}}$. By Lemma \ref{product coinsides with the stabilizer}, there exist $h \in G_{q(\Delta)}^0$ and $h^{\prime} \in G_{\Delta}$ such that $f=h h^{\prime}$. So the element $h$ also exchanges the two elements of $\Delta^{\{2\}}$, and since $h \in G_{q(\Delta)}^0$ there exists a nonempty open subset $V \subset \Lambda_S(G)$ such that $q(\Delta) \notin \overline{V}$ and $h \in \rist_G(V)$.

Now let $U$ be an arbitrary nonempty open subset of $\Lambda_S(G)$. Since the action of $G_{\Delta}$ on $\Lambda_S(G)$ is quasi-parabolic, by Lemma \ref{Gen type actions admit compressible open sets} there exists $l \in \mathcal{L}_S(G_{\Delta})$ such that $lV \subset U$ and, thus, $g=l hl^{-1} \in \rist_G(U)$. Note that $g$ also exchanges the two elements of $\Delta^{\{2\}}$ since $l \in G_{\Delta}$.
\end{proof}

We complete the proof of Proposition \ref{Stab Delta is not quasi-parabolic}.

\begin{proof}[Proof of Proposition \ref{Stab Delta is not quasi-parabolic}]
    Let $\Delta^{\{2\}} \in \Omega^{\{2\}}$ and $\Delta'^{\{2\}} \in \mathcal{P}_\Delta$ distinct from $\Delta^{\{2\}}$, $V \subset \Lambda_S(G)$ and $h \in \rist_G(V)$ be as in Lemma \ref{Delta and Delta'}, and let $\widetilde{\Delta}^{\{2\}} \in \Omega^{\{2\}}$ such that $\widetilde{\Delta}^{\{2\}} \cap \Delta^{\{2\}}$ and $\widetilde{\Delta}^{\{2\}} \cap \Delta'^{\{2\}}$ are nonempty. Say $\Delta^{\{2\}}=\{\omega, \omega'' \}, \Delta'^{\{2\}}= \{\omega', \widetilde{\omega}\}$ and $\widetilde{\Delta}^{\{2\}}=\{\widetilde{\omega}, \omega\}$, where $\omega, \omega', \widetilde{\omega}$, and $\omega'' \in \Omega$ are distinct. Choose a nonempty open subset $W \subset \Lambda_S(G)\setminus V$ with $q(\Delta) \notin \overline{W}$. We apply Lemma \ref{L:3.13} and find $g \in \rist_G(W)$ such that $g\widetilde{\omega}=\omega$ and $g\omega=\widetilde{\omega}$. Note that $g$ fixes $q(\Delta)$, and hence $g$ must preserve the partition $\mathcal{P}_{\Delta}$. Since $g$ exchanges two elements of $\Delta^{\{2\}}$ and $\Delta'^{\{2\}}$, it follows that $g$ actually exchanges the blocks $\Delta^{\{2\}}$ and $\Delta'^{\{2\}}$.
    Thus, we have
    \begin{equation}\label{contradiction equation}
        hg\Delta^{\{2\}}=h\Delta'^{\{2\}} \neq \Delta^{\{2\}}.
    \end{equation}
    Since $g$ and $h$ are supported in disjoint open subsets of $\Lambda_S(G)$, they commute. Therefore,
    $$
    hg\Delta^{\{2\}}=gh\Delta^{\{2\}}=g\Delta'^{\{2\}}=\Delta^{\{2\}},
    $$ which contradicts (\ref{contradiction equation}).
\end{proof}

The next lemma completes the proof outlined after Proposition \ref{non-elementary action}.

\begin{lem}\label{The action of a stabilizer of a point is quasi-parabolic}
    Assume that $G$ acts faithfully and $3$-transitively on a set $\Omega$. Then for any $\omega \in \Omega$, the action of a subgroup $G_\omega$ on $S$ is quasi-parabolic.
\end{lem}

\begin{proof}
    Fix any $\omega \in \Omega$. By Proposition \ref{non-elementary action}, the action of $G_\omega$ on $S$ is non-elementary. Therefore, it suffices to prove that it is not of general type. For the sake of contradiction, assume that the action of $G_\omega$ on $S$ is of general type. The action of $G$ on $\Omega$ is $3$-transitive. Therefore, the action of $G_\omega$ on $\Omega \setminus \{\omega\}$ is $2$-transitive. Let $\omega_1 \in  \Omega \setminus \{\omega\}$ and $\Delta=(\omega_1, \omega)$. We apply Proposition \ref{non-elementary action} to $G_\Delta \leq G_\omega$ and obtain that $G_\Delta$ is itself non-elementary. By Proposition \ref{Stab Delta is not of general type}, the action of $G_\Delta$ on $S$ cannot be of general type, and by Proposition \ref{Stab Delta is not quasi-parabolic}, it cannot be quasi-parabolic. This leads to a contradiction. 
\end{proof}

Now we are ready to prove Theorem \ref{thm about conjugate orbits}.

\begin{proof}[Proof of Theorem \ref{thm about conjugate orbits}]
By Lemma \ref{The action of a stabilizer of a point is quasi-parabolic}, the action of $G_\omega$ on $S$ is quasi-parabolic for any $\omega \in \Omega$. Let $q(\omega) \in \Lambda_S(G_\omega)$ be the unique fixed point of $G_\omega$. Consider the map $q \colon \Omega \rightarrow \partial S$, given by $\omega \mapsto q(\omega)$. This map is $G$-equivariant. For the sake of contradiction, assume that $q$ is not injective, i.e., there exists $\omega_1, \omega_2 \in \Omega$, such that $\omega_1 \neq \omega_2$, but  $q(\omega_1)=q(\omega_2)$. The action of $G$ is $2$-transitive on $\Omega$, so for any distinct $\omega_1', \omega_2' \in \Omega$ there exists $g \in G$ such that $g\omega_1=\omega_1'$ and $g\omega_2=\omega_2'$. Therefore, $G_{\omega_1}=g^{-1}G_{\omega_1'}g$ and $G_{\omega_2}=g^{-1}G_{\omega_2'}g$. Thus, $q(\omega_1')=gq(\omega_1)=gq(\omega_2)=q(\omega_2')$. Then the map $q$ is constant, and it follows that $G$ fixes a point in $\Lambda_S(G)$, which contradicts the assumption that the action of $G$ on $S$ is of general type. 
Therefore, $q$ is injective, and the $G$-action on $\Omega$ is conjugate to the action on $\mathcal{O}=q(\Omega)$.
\end{proof}

\begin{proof}[Proof of Corollary \ref{td is leq 3}]
The group $G$ admits a general type action on a hyperbolic space $S$ such that the induced action of $G$ on $\Lambda_S(G)$ is faithful and not topologically free. Thus, by Theorem \ref{thm about conjugate orbits}, if $G$ acts faithfully and 3-transitively on any set $\Omega$, then this action is conjugate to the action of $G$ on some $G$-orbit $\mathcal{O} \subset \Lambda_S(G)$.  By Lemma \ref{set of fixed points of loxodromic elements is dense}, the set of fixed points of elements of $\mathcal{L}_S(G)$ is dense in $\Lambda_S(G)$. Thus, since $\mathcal{O}$ is $G$-invariant, by Lemma \ref{north-south}, it is also dense in $\Lambda_S(G)$. Therefore, by Proposition \ref{not 4-transitivity}, the action of $G$ on $\mathcal{O}$ is not topologically $4$-transitive. Thus, by Lemma \ref{transitivity and topological transitivity}, it is at most $3$-transitive. Therefore, $G$ does not admit any faithful $4$-transitive action on any set $\Omega$, and the transitivity degree of $G$ is at most $3$.
\end{proof}

\begin{proof}[Proof of Corollary \ref{topological 3-transitivity}]
Since $G$ is faithful weakly hyperbolic and not lim-free, it admits a general type action on a hyperbolic space $S$ such that the induced action of $G$ on $\Lambda_S(G)$ is faithful and not topologically free. Moreover, by Theorem \ref{main_theorem}, for any general type action of $G$ on any hyperbolic space, such that the induced action on the limit set is faithful, the induced action on the limit set is not topologically free. Since the transitivity degree of $G$ is $3$, there exists a set $\Omega$ such that the $G$-action on $\Omega$ is faithful and $3$-transitive. Then, by Theorem \ref{thm about conjugate orbits}, the $G$-action on $\Omega$ is conjugate to the $G$-action on some orbit $\mathcal{O} \subset \Lambda_S(G)$. Since the action of $G$ on $S$ is of general type, $\mathcal{O}$ is dense in $\Lambda_S(G)$. Therefore, the action of $G$ on $\Lambda_S(G)$ is topologically $3$-transitive. \end{proof}

\begin{cor}\cite[Theorem 1.4]{BM}\label{conjugate orbits for trees}
Let $T$ be a simplicial tree. Suppose that the action of $G \leq \Aut(T)$ on $T$ is minimal and of general type and that the action of $G$ on $\partial T$ is not topologically free. Assume that $G$ acts faithfully and 3-transitively on a set $\Omega$. Then there exists a $G$-orbit $\mathcal{O} \subset \partial T$ such that the actions of $G$ on $\Omega$ and on $\mathcal{O}$ are conjugate. 
\end{cor}
\begin{proof}
    It is sufficient to show that $G$ satisfies the conditions of Theorem \ref{thm about conjugate orbits}. This was already shown in the proof of Corollary \ref{previous results about MIF in trees}.
\end{proof}

\subsection{Examples and necessity of assumptions}\label{examples}

\begin{ex}\label{Verbal products}
{\bf Verbal products.}
This example provides a family of non-faithful weakly hyperbolic groups that satisfy a mixed identity.

Let $V=\left\{f_\nu\left(x_1 \ldots x_{n_\nu}\right): \nu \in I\right\}$ be a subset of words in an infinitely generated free group. A {\it verbal $V$-subgroup} $V(G)$ of a group $G$ is a subgroup generated by all possible values of all words of $V$, when $x_1, x_2, \ldots$ run through the entire group $G$ independently of each other. A verbal $V$-subgroup $V(G)$ is normal in $G$. Let $G_i, i \in I$ be some collection of groups and let $F = *_{i\in I}G_i$. Let $\pi \colon F  \longrightarrow \prod_{i\in I}G_i$ be the natural epimorphism. Then the {\it Cartesian subgroup} $C < F$ is the kernel of $\pi$. Let $V(F)$ be a verbal $V$-subgroup of $F$. Then the {\it verbal product} $*_V G_i$ of groups $G_i$ with respect to $V$ is $F/(V(F) \cap C)$.
\end{ex}
\begin{lem}\label{Verbal products are not MIF}

Let $A$ and $B$ be nontrivial groups, and let $V$ be a nonempty subset of words in a free group $F(x_1, \cdots, x_k)$. Then the group $G=A *_V B$ satisfies a mixed identity.
\end{lem}
\begin{proof}
Let $v=v(x_1, x_2, x_3, \cdots, x_k)$ be a nonempty word in $V$, and let $a \in A$ and $b \in B$ be non-trivial elements of $G$. Consider the element  
$$w(a,b,x_1, x_2, x_3, \cdots, x_k)=v([a^{x_1},b], [a^{x_2}, b], [a^{x_3},b], \cdots, [a^{x_k}, b]) \in G*F(x_1, \cdots, x_k).$$ For any $i \in \{1, \dots, k\}$ and values of $x_i$ in $G$, elements $[a^{x_i}, b]$ are in the Cartesian subgroup of $A*B$, thus $v([a^{x_1},b], [a^{x_2}, b], [a^{x_3},b], \cdots, [a^{x_k}, b]) \in V(A*B)$. Therefore, $w(a,b,g_1, g_2, g_3, \cdots, g_k)=1$ in $G$ for any values $g_1, \cdots, g_k \in G$. Since $a, b$ are non-trivial in $G$, $w(a,b,x_1, x_2, x_3, \cdots, x_k)$ is non-trivial in $G*F(x_1, \cdots, x_k)$.
\end{proof}

\begin{prop}\label{WH not FWH groups}
Let $A$, $B$ be nontrivial groups, and let at least one of $A$ or $B$ be weakly hyperbolic. Let $V$ be a nonempty subset of words in a free group $F(x_1, \cdots, x_k)$. Then $G=A *_V B$ is a weakly hyperbolic group that is not faithful weakly hyperbolic.
\end{prop}
\begin{proof}
Without loss of generality, assume that $A$ is weakly hyperbolic. There exists a surjective homomorphism of $G$ onto $A$, thus $G$ is also weakly hyperbolic. Let $\phi(C)$ be the image of the Cartesian subgroup $C < A*B$ under the homomorphism $\phi \colon A*B \rightarrow G$.

If $\phi(C)=1$, then $G$ is a direct product of two non-trivial groups, so it is not MIF and $\Mon(G)=1$. Thus, by Proposition \ref{monolithic groups}, $G$ cannot be faithful weakly hyperbolic. 

Assume now that $\phi(C) \neq 1$. Since $\phi$ is surjective, $\phi(C)$ is normal in $G$. Therefore, by Lemma \ref{dichotomy}, for any general type action of $G$ on any hyperbolic space $S$, the subgroup $\phi(C)$ is either elliptic or of general type. Since $V(\phi(C))=\{1\}$, $\phi(C)$ satisfies some non-trivial identity $v=1$, $v \in V$ and cannot contain a non-abelian free subgroup. So, $\phi(C)$ is elliptic, and since it is normal in $G$, $\phi(C) < E_S(G)$ by the maximality of $E_S(G)$. Therefore, $E_S(G)$ is non-trivial, and $G$ is not a faithful weakly hyperbolic group. 
\end{proof}

\begin{ex}{\bf Non-lim-free HNN extensions and amalgams.} The following proposition provides restrictions on associate subgroups of $\HNN$-extensions and amalgamated products when the action on the boundary of the Bass-Serre tree is not topologically free.
\end{ex}

\begin{defn}
Let $H$ be a group and let $\phi \colon C \rightarrow \phi(C)$ be an isomorphism between subgroups of $H$. Then {\it $\HNN$-extension} associated to $(H, C, \phi)$ is the group defined by the following presentation
$$
\HNN(H, C, \phi):=\langle H, t \mid t^{-1}ct=\phi(c) \text{ for all } c \in C \rangle,
$$
where $t$ is an extra generator, not belonging to $H$, called {\it a stable letter}.
\end{defn}
Recall that it is called {\it ascending} if one of the subgroups $C, \phi(C)$ is equal to $H$.

\begin{defn}
Let $l_1: C \rightarrow A$ and $l_2: C \rightarrow B$ be injective group homomorphisms. We will denote by $C_j$ the image of $l_j$ and by $\phi: C_1 \rightarrow C_2$ the isomorphism sending $l_1(c)$ to $l_2(c)$ for all $c \in C$. The {\it free product with amalgamation (or amalgam for short)} associated to $(A, B, C, \phi)$ is the group defined by the following presentation:
$$
A *_{C} B:=\left\langle A, B \mid l_1(c)=l_2(c) \text{ for all } c \in C \right\rangle = \left\langle A, B \mid c=\phi(c) \text{ for all } c \in C_1 \right\rangle.
$$
\end{defn}
Recall that such an amalgam is said to be {\it non-trivial} if $A, B \neq C_j$ for $j=1,2$, and {\it non-degenerate} if moreover $\left[A : C_1\right] \geq 3$ or $\left[B : C_2\right] \geq 3$.

Given an $\HNN$-extension $G=\HNN(H, C, \phi)$ or an amalgam $G=A*_{C}B$, we can construct a Bass-Serre tree $T$ associated to $G$. The group $G$ acts on $T$
transitively on the edges of $T$. Stabilizers of edges are conjugates of $C$ or $\phi(C)$. The action of $G$ on $T$ is always minimal. We refer the reader to \cite{Ser} for more information. If $\HNN(H, C, \phi)$ is non-ascending or $G=A*_{C}B$ is non-degenerate, then the action of $G$ on $T$ is of general type (see, for example, \cite[Sections 2.5, 2.6]{FMMS}).

\begin{prop}
Let $G=\HNN(H, C, \phi)$ be a non-ascending HNN-extension or  a non-degenerate amalgam $G=A*_{C}B$. If the induced action of $G$ on the boundary $\partial T$ of the Bass-Serre tree is faithful but not topologically free, then $C$ is non-abelian. 
\end{prop}
\begin{proof}
 By Lemma \ref{Not_top_free actions are micro-supported}, the action of $G$ on $\partial T$ is micro-supported, and, by Lemma \ref{non-abelian rigid stabilizers}, $\rist_G(U)$ is non-abelian for every nonempty open subset $U \subseteq \partial T$.  Take an arbitrary open $U$ that is not dense in $\partial T$. Consider two adjacent vertices $v_1$ and $v_2$ in $T$, such that the ends of infinite rays from these vertices are in $\partial T \setminus U$. Then $\rist_G(U)$ fixes both $v_1$ and $v_2$; thus, stabilizes the edge $e$ incident to $v_1$ and $v_2$. Since stabilizers of edges are conjugate to $C$, there exists $g \in G$ such that $\rist_G(U) \leq C^g$. Therefore, $C$ is non-abelian.
\end{proof}

\begin{cor}
The action of a non-solvable Baumslag-Solitar group $BS(m,n)$ on the boundary of its Bass-Serre tree is not topologically free if and only if $|m|=|n|$.
\end{cor}

\begin{ex}
{\bf Solvable Baumslag-Solitar groups.}
This example shows that the conclusion of Theorem \ref{real main theorem} does not hold if we take a quasi-parabolic action instead of a general-type action in the assumptions.

The Baumslag-Solitar group $BS(1,2)$ admits a quasi-parabolic action on its Bass-Serre tree $T_{BS}$ since it is an ascending HNN extension (see \cite[Section 2.5]{FMMS}), such that the induced action on $\Lambda_{T_{BS}}(BS(1,2))=\partial T_{BS}$ is topologically free \cite[Lemma 8.6]{FMMS}. However, $BS(1,2)$ is solvable of order $3$, so it satisfies the identity $$
[[[x_1,x_2],[x_3,x_4]],[[x_5,x_6],[x_7,x_8]]]=1 .
$$
\end{ex}

\begin{ex}
{\bf Topologically free action on the boundary.} This example shows why in Theorem \ref{real main theorem} we must consider topologically free actions on the limit set, not the whole boundary as an equivalent condition. A similar example is given in \cite[Proposition 9.2]{FMMS}, where they discuss the necessity of minimality of the action on the tree in their case, which exactly gives faithfulness of the action on the limit set.

Let $G = F_2 \times \ZZ$, where $F_2$ is a free group of rank $2$. The action of $F_2$ on its Cayley graph $T$ is of general type. Let $E$ be a countable bundle of rays from a unique point $o$, enumerated by integers. The group $\ZZ$ acts on $E$ as follows: $n\cdot r_m = r_{n+m}$ for every $n \in \ZZ$ and every ray $r_m \subset E$. This action of $\ZZ$ on $E$ is elliptic since it fixes a point $o$. Let $S$ be a space, constructed by gluing over each vertex of $T$ a copy of $E$ at its fixed point $o$. Then $G$ acts on $S$ by isometries, where elements from $F_2$ act on a copy of $T$ and elements of $\ZZ$ act on copies of $E$ in $S$. This action is of general type, and the induced action on $\partial S$ is topologically free. However, it is not topologically free (and not even faithful) on $\Lambda_S(G) = \partial T$, since elements $(1, n) \in F_2 \times \ZZ$, $n\neq 0$ fix $\partial T$. Obviously, $G$ satisfies a nontrivial mixed identity $[f^x, n]=1$, where $f \in F_2$, $x$ is a free variable, and $n \in \ZZ\setminus\{0\}$.  
\end{ex}

\begin{ex}{\bf Monolithic lim-free groups}. This example shows that the converse of Proposition \ref{monolithic groups} is not true, i.e., there exist lim-free weakly hyperbolic (hence MIF) groups with infinite simple monolith.

Consider $\PSL_2(\CC)$, $\PSL_2(\RR)$, $\PSL_2(\mathbb{Q})$, or any $\PSL_2(F)$, where $F$ is a field of $char=0$ with cardinality less or equal to continuum. By the proof of Corollary \ref{generalized Tits Alternative}, they are lim-free, and it is well known that they are simple (see \cite{Dix}), so coincide with their monolith.
\end{ex}
\begin{ex}{\bf Non-faithful weakly hyperbolic MIF groups.}\label{example H}
\end{ex}
\begin{prop}\label{construction of MIF WH not FWH groups}
Every countable group can be embedded into a finitely generated MIF weakly hyperbolic group, which is not faithful.
\end{prop}
\begin{proof} Let $K$ be an arbitrary countable group. By the Higman-Neumann-Neumann Embedding Theorem \cite{HNN}, $K$ can be embedded in a finitely generated group $G$. By the construction in \cite{HNN}, $G$ is a non-ascending HNN-extension; therefore, it is weakly hyperbolic. Then by \cite[Thm 1.1]{HL}, there exists a finitely generated group $H$, such that:
\begin{enumerate}
    \item $G \leq H \leq \operatorname{Sym}(G \times \ZZ)$;
    \item There is a short exact sequence:

$$
1 \longrightarrow N \longrightarrow H \longrightarrow(G \times \ZZ) \longrightarrow 1,
$$

where $N$ is locally finite;
\item  $H$ is MIF.
\end{enumerate}
The group $H$ maps surjectively on $G \times \ZZ$, and $G \times \ZZ$ is weakly hyperbolic, since $G$ is weakly hyperbolic. Therefore, $H$ also admits a general type action on the same hyperbolic space as $G$. Note that the subgroup $N$ is non-trivial; otherwise $H$ would be isomorphic to $G \times \ZZ$, which is impossible since $H$ is MIF and  $G \times \ZZ$ has a non-trivial center. Since $N$ is locally finite, it consists of torsion elements of $H$. Since $N$ is normal in $H$, it cannot be a parabolic subgroup with respect to any action on any hyperbolic space, so it is always elliptic, and, by Proposition \ref{elliptic radical is maximal}, it should be in an elliptic radical of $H$ with respect to any general type action. Therefore, any elliptic radical of $H$ is non-trivial, and $H$ is not faithful weakly hyperbolic. 
\end{proof}

\begin{ex} 
{\bf Finitely presented simple MIF groups.}\label{ex BM}
\end{ex}
{\bf Kac-Moody groups.} We partially answer Question 3.6 from \cite{BFFHZ}.

\begin{proof}[Proof of Proposition \ref{Kac-Moody}]
    Let $\Gamma(F)$ be a finitely presented Kac-Moody group over a finite field, constructed in \cite{CR}, such that its quotient by the finite center $G(F)=\Gamma(F)/Z(\Gamma(F))$ is an infinite finitely presented simple group. We want to show that $G(F)$ is lim-free weakly hyperbolic. First, we will show that $G(F)$ is weakly hyperbolic.

    The group $G(F)$ is a nonuniform lattice in a product of two buildings.  Since $F$ is finite, the associated building is locally finite; hence its Davis realization $X$ is a proper $\CAT(0)$ space. Since $F$ is a field, by the proof of \cite[Corollary 1.3 (page 19, line 1 of the proof)]{CF}, $G(F)$ acts Weyl-transitively on each of the two buildings.
    By \cite[Proposition 5.3]{CF}, the action of $G(F)$ on one of these buildings $X$ contains two rank-one elements $g_1$ and $g_2$ with disjoint limit points in the boundary of $X$. 

    By \cite[Theorem B]{PSZ}, there exists $\delta$ such the curtain model $X_D$ associated to $X$ is $\delta$-hyperbolic and $\Isom(X) \leq \Isom(X_D)$. By \cite[Theorem C]{PSZ}, two rank-one elements $g_1$ and $g_2$ of the action of $G(F)$ on $X$ become loxodromic elements of the action of $G(F)$ on $X_D$. Finally, by \cite[Theorem L]{PSZ}, the boundary $\partial X_D$ embeds as an $\Isom$-invariant subset of $\partial X$; thus, the four limit points of rank-one elements $g_1$ and $g_2$ in $\partial X$ must also be four distinct points when looking at the corresponding limit points on $\partial X_D$. Therefore, the action of $G(F)$ on $X_D$ is of general type. Since $G(F)$ is simple, $E_{X_D}(G(F))= \{1\}$, and the action of $G(F)$ on $\Lambda_{X_D}(G(F))$ is faithful. Note that the space $X_D$ is not necessarily geodesic, though, by \cite[Proposition N]{PSZ}, there is a geodesic hyperbolic space $E(X_D)$, called the {\it injective hull}, such that the canonical embedding $X_D \rightarrow E(X_D)$ is an $\Isom(X_D)$-equivariant quasi-isometry. In particular, it induces an $\Isom(X)$-equivariant homeomorphism between the Gromov boundaries of $X_D$ and $E(X_D)$. Thus, $g_1$ and $g_2$ act loxodromically on $E(X_D)$ and have disjoint pairs of fixed points. Moreover, the elliptic radical
$$
E_{E(X_D)}(G(F))=E_{X_D}(G(F))=\{1\}.
$$
Thus, we can pass to the space $E(X_D)$ to conclude that $G(F)$ is faithful weakly hyperbolic.

    Now, assume that the action of $G(F)$ on $\Lambda_{E(X_D)}(G(F))$ is not topologically free. Then, by Corollary \ref{infinite direct products}, the rigid stabilizers of nonempty open subsets of $\Lambda_{E(X_D)}(G(F))$ are infinite and contain infinite direct sums of other rigid stabilizers. Assume that a rigid stabilizer of some nonempty non-dense open subset $U$ of $\Lambda_{E(X_D)}(G(F))$ has an element of order $1 < n < \infty$. Then Lemma \ref{wreath products in rigid stabilizers} allows us to construct elements of arbitrarily large finite order in $\rist_{G(F)}(U)$. Since the center $Z(\Gamma(F))$ is finite, this contradicts \cite[Corollary 1.3]{Cap}. Therefore, the rigid stabilizers of nonempty non-dense open subsets in the limit set $\Lambda_{E(X_D)}(G(F))$ are torsion-free. Thus, $G(F)$ contains an infinite direct sum $P$ of infinite cyclic subgroups. By \cite[Theorem 64]{BD}, the asymptotic dimension of $G(F)$ is at least the supremum of the asymptotic dimensions of its finitely generated subgroups. For every $n \in \NN$, a subgroup $P$ contains a direct sum $P_n$ isomorphic to $\ZZ^n$. Since the asymptotic dimension of $\ZZ^n$ is $n$, we obtain that the asymptotic dimension of $G(F)$ is infinite. This yields a contradiction with the fact that the asymptotic dimension of $G(F)$ is bounded from above by the asymptotic dimension of the product of two buildings it acts on \cite[last paragraph]{SWZ}, which is finite by \cite{DS}. Therefore, the action of $G(F)$ on $\Lambda_{E(X_D)}(G(F))$ is topologically free and, by Theorem \ref{main_theorem}, $G(F)$ is MIF. By Corollary \ref{selfless groups}, the reduced $C^*$-algebra of $G(F)$ is selfless.
    \end{proof}

{\bf Burger-Mozes group.} In \cite{BM97}, the finitely presented simple Burger-Mozes group $G$ was constructed. It acts by isometries on a locally finite tree $T$ such that the induced action on $\partial T$ is $2$-transitive (\cite[Remark 5]{BM97}). Therefore, the action of $G$ on $T$ is of general type. By construction, $G$ acts properly and cocompactly by isometries on the product of two locally finite trees. Hence, it has finite asymptotic dimension. So, the argument of Proposition \ref{Kac-Moody} can be applied to $G$ to conclude that $G$ is MIF, and the reduced $C^*$-algebra of $G$ is selfless.   
The fact that $G$ is MIF was first observed in \cite[Remark 3.5]{BFFHZ}.

{\bf Amir-Lazarovich groups.}
Very recently, in \cite{AL}, M. Amir and N. Lazarovich constructed simple finitely presented groups acting on products of two Davis complexes: a $c$-regular tree with $c \geq 6$ and the right-angled Davis complex of an Odd graph. By \cite[Theorem 5.2]{AL}, the action of such groups on a tree is vertex-transitive and therefore is of general type. Thus, the rest of the argument of Proposition \ref{Kac-Moody} also applies, and such groups are MIF and have selfless reduced $C^*$-algebras.

\begin{rem}
In \cite{TMW} a family of lattices acting on two-dimensional Euclidean buildings was introduced, and recently it was proved that this construction also produces a family of infinite simple groups \cite{LW}. Unfortunately, the reasoning from Proposition \ref{Kac-Moody} does not apply to this family, and we do not know whether they are mixed identity free or not.  
\end{rem}

\vspace{5mm}

\noindent \textbf{Ekaterina Rybak: } Department of Mathematics, Vanderbilt University, Nashville, TN, U.S.A.

\noindent E-mail: \emph{ekaterina.rybak@vanderbilt.edu}


\begin{thebibliography}{99}

\bibitem[ABDHMW]{Spencer}
T. Aougab, H. Bray, S. Dowdall, H. Hoganson, S. Maloni, B. Whitfield, Constructing reducibly geometrically finite subgroups of the mapping class group. \emph{Groups Geom. Dyn.}, 2026.

\bibitem[AG]{AG}
N. Avni, T. Gelander, Mixed identities in linear groups -- effective version,  \emph{arXiv:2510.03492v1}.

\bibitem[AGKEP]{AGKEP}
T. Amrutam, D. Gao, S. Kunnawalkam Elayavalli, G. Patchell, Strict comparison in reduced group $C^*$-algebras. \emph {Invent. math.} \textbf{242} (2025), 639--657.

\bibitem[AL]{AL}
M. Amir, N. Lazarovich, Simple lattices in products of Davis complexes, \emph{arXiv:2605.09493}.

\bibitem[AMS]{AMS}
R. Aoun, P. Mathieu and \c C. Sert, Random walks on hyperbolic spaces: second order expansion of the rate function at the drift, \emph{J. \'Ec. polytech. Math.} \textbf{10} (2023), 549--573.

\bibitem[AS]{AS}
R. Aoun and \c C. Sert, Random walks on hyperbolic spaces: concentration inequalities and probabilistic Tits alternative, \emph{Probab. Theory Related Fields} {\bf 184} (2022), no.~1-2, 323--365.


\bibitem[Bas]{Bas}
H. Bass, Covering theory for graphs of groups, \emph{J. Pure Appl. Algebra} {\bf 89} (1993), no.~1-2, 3--47.

\bibitem[BD]{BD}
G. Bell, A. Dranishnikov, Asymptotic dimension, \emph{Topology Appl.}, \textbf{155} (2008), 1265--1296.

\bibitem[Bea]{Bea}
A. Beardon, The geometry of discrete groups. Graduate Texts in Mathematics, \emph{Springer}, {\bf 91}, New York etc., 1995.

\bibitem[BEH]{BEH}
C. Bleak, L. Elliott, J.~T. Hyde, Sufficient conditions for a group of homeomorphisms of the Cantor set to be two-generated, \emph{J. Inst. Math. Jussieu} {\bf 23} (2024), no.~6, 2825--2858.

\bibitem[Bes]{Bes}
M. Bestvina. $\mathbb{R}$-trees in topology, geometry, and group theory. In \emph{Handbook of Geometric Topology}, pages 55-91. North Holland, 2001.

\bibitem[BFFHZ]{BFFHZ}
J. Belk, F. Fournier-Facio, J. Hyde, M. C. B. Zaremsky, Boone-Higman embeddings of $Aut(F_n)$ and mapping class groups of punctured surfaces,  arXiv:2503.21882v2.

\bibitem[BFGS]{BFGS}
S.Balasubramanya, F. Fournier-Facio, A.Genevois, Property (NL) for hyperbolic group actions, with Appendix by A.Sisto, {\it arXiv:2212.14292}.

\bibitem[BH]{BH}
H. Baik, W. Jang, On the kernel of actions on asymptotic cones, \emph{Journal of Group Theory}, \textbf{28} (2025) 753--809.

\bibitem[BH99]{BH99} M.R. Bridson, A. Haefliger,
Metric spaces of non-positive curvature.
\emph{Grundlehren der Mathematischen Wissenschaften} [Fundamental Principles of Mathematical Sciences] \textbf{319}. \emph{Springer-Verlag, Berlin}, 1999.

\bibitem[BM]{BM}
A. Le Boudec, N. Matte Bon, Triple transitivity and non-free actions in dimension one, \emph{J. Lond. Math. Soc.} \textbf{105} (2022), 2, 884--908.

\bibitem[BM97]{BM97}
M. Burger and S. Mozes, Finitely presented simple groups and products of trees, $C$. R. \emph{Acad. Sci. Paris Sér.} I Math., (1997) \textbf{324} 747--752.

\bibitem[BMSS]{BMSS}
A. Boulanger, P. Mathieu, C. Sert, A. Sisto, Large deviations for random walks on Gromov-hyperbolic spaces, \emph{Ann. Sci. \'Ec. Norm. Sup\'er.} (4) {\bf 56} (2023), no.~3, 885--944.

\bibitem[BS07]{BS07}
Sergei Buyalo and Viktor Schroeder. Elements of asymptotic geometry. \emph{EMS Monographs in Mathematics. European Mathematical Society (EMS)}, Zurich, 2007.

\bibitem[BS]{BS}
H. Bradford, A. Sisto, Non-solutions to mixed equations in acylindrically hyperbolic groups coming from random walks, \emph{arXiv:2504.15456v1}.


\bibitem[CDP]{french book}
M. Coornaert, T. Delzant, A. Papadopoulos, Géométrie et théorie des groupes, {\emph Lecture Notes in Mathematics}, {\bf 1441}, Springer-Verlag, Berlin, 1990.

\bibitem[Cap]{Cap}
P-E. Caprace, A uniform bound on the nilpotency degree of certain subalgebras of Kac–Moody algebras, \emph{J. of Algebra} \textbf{317}(2) (2007), 867--876.

\bibitem[CF]{CF}
P-E. Caprace, K. Fujiwara, Rank-One Isometries of Buildings and Quasi-Morphisms of Kac–Moody Groups. \emph{Geom. Funct. Anal.} \textbf{19} (2010), 1296--1319.

\bibitem[Cho]{Cho}
I. Choi, Central limit theorem and geodesic tracking on hyperbolic spaces and Teichm\"uller spaces, \emph{Adv. Math.}, {\bf 431} (2023), 68.

\bibitem[Cou16]{Cou16}
R. Coulon, Partial periodic quotients of groups acting on a hyperbolic space, \emph{Ann. Inst. Fourier (Grenoble)}, {\bf 66} (2016), 1773--1857.

\bibitem[CR]{CR}
P.-E. Caprace, B. Rémy, Simplicity and superrigidity of twin building lattices, \emph{Invent. Math.}, \textbf{176}(2009), 169--221.

\bibitem[CRW]{CRW}
 P. Caprace, C. Reid, G. Willis, Locally normal subgroups of totally disconnected groups. Part II: Compactly generated simple groups, \emph{Forum Math. Sigma} \textbf{5} (2017), no. e12, 89.

\bibitem[DGO]{DGO}
F. Dahmani, V. Guirardel, and D. Osin, Hyperbolically embedded subgroups and rotating families in groups acting on hyperbolic spaces, \emph{Mem. Amer. Math. Soc.}, \textbf{245} (2017), v+152.

\bibitem[Dix]{Dix}
L. E. Dickson, Theory of linear groups in an arbitrary field, {\emph Trans. Amer. Math. Soc.} \textbf{2} (1901), 363--394.

\bibitem[DSU]{DSU}
T. Das, D. Simmons,  M. Urbanski, Geometry and dynamics in Gromov hyperbolic metric spaces: With an emphasis on non-proper settings, Mathematical Surveys and Monographs \textbf{218}, 2017.



\bibitem[FMMS]{FMMS}
P. Fima, F. Le Ma\^{i}tre, S. Moon, Y. Stalder, A characterization of high transitivity for groups acting on trees, \emph{Discrete Anal.} \textbf{8} (2022), 63.

\bibitem[GGS]{GGS}
T. Gelander, Y. Glasner, G. Soifer, Maximal subgroups of countable groups: a survey. In Dynamics, geometry, number theory: the impact of Margulis on modern mathematics, \emph{University of Chicago Press} (2022), 169--210.


\bibitem[GOR]{GOR}
G. Goffer, D.~V. Osin and E. Rybak, Frattini subgroups of hyperbolic-like groups, \emph{J. Algebra} {\bf 658} (2024), 22--37.

\bibitem[Gro87]{Gro87}
M. Gromov, Hyperbolic groups, Essays in Group Theory, \emph{MSRI Series} \textbf{8}, (S.M. Gersten, ed.), \emph{Springer}, 1987, 75--263.

\bibitem[GR]{GR}
A. Garrido, C. D. Reid, Compressible subgroups and simplicity, {\it arxiv:2412.18891}

\bibitem[DS]{DS}
J. Dymara, T. Schick, Buildings have finite asymptotic dimension, \emph{Russ. J. Math. Phys.}, \textbf{16} (2009), 409-–412.

\bibitem[GTT22]{GTT22}
I. Gekhtman, S.~J. Taylor and G. Tiozzo, Central limit theorems for counting measures in coarse negative curvature, \emph{Compos. Math.} {\bf 158} (2022), no.~10, 1980--2013.


\bibitem[Ham]{Ham}
M. Hamann, Group actions on metric spaces: fixed points and free subgroups, \emph{Abh. Math. Semin. Univ. Hambg.} \textbf{87} (2017), 245--263.

\bibitem[HL]{HL}
J. Hyde, Y. Lodha, Embeddings into highly transitive and mixed identity free groups, \emph{arXiv:2509.09788}, 2025

\bibitem[HNN]{HNN}
G. Higman, B.~H. Neumann and H. Neumann, Embedding theorems for groups, \emph{J. London Math. Soc.} {\bf 24} (1949), 247--254.

\bibitem[HO13]{HO13}
M. Hull, D. Osin, Conjugacy growth of finitely generated groups, \emph{Advances in Mathematics.} \textbf{235} (2013) (1), 361--389.

\bibitem[HO16]{OH}
M. Hull, D. Osin, Transitivity degrees of countable groups and acylindrical hyperbolicity, \emph{Israel J. Math.} \textbf{216} (2016), no.1, 307--353.

\bibitem[Hor]{Hor}
C. Horbez, Central limit theorems for mapping class groups and ${\rm Out}(F_N)$, Geom. Topol. {\bf 22} (2018), no.~1, 105--156.

\bibitem[HV]{HV}
P. de la Harpe, Alain Valette. La propriété $(T)$ de Kazhdan pour les groupes localement compacts (avec un appendice de Marc Burger).\emph{ Astérisque, No. 175. Société Mathématique de France}, Paris, 1989. With an appendix by M. Burger.

\bibitem[Jac]{Jac}
B. Jacobson, Algebraic subgroups of acylindrically hyperbolic groups, \emph{J. of Algebra}, \textbf{476} (2017), 113--133.

\bibitem[KM]{KM}
O. Kharlampovich, A. Myasnikov, Definable sets in a hyperbolic group, \emph{Internat. J. of Algebra and Comput.} \textbf{23} (2013), 91--110.

\bibitem[KS]{KS}
M. Kapovich, P. Sardar, Trees of hyperbolic spaces, \emph{Mathematical surveys and monographs}, {\bf 282}, 2024.

\bibitem[La]{La}
S. Lang, Algebra. Volume 1. Revised third edition, \emph{Graduate Texts in Mathematics} {\bf 211}, New York, NY: \emph{Springer} (2002)

\bibitem[LS]{LS}
R. Lyndon, P. Schupp, Combinatorial Group Theory,\emph{Classics in Mathematics, Springer Berlin, Heidelberg} (2001)

\bibitem[LW]{LW}
J. Lécureux, S. Witzel, The Normal Subgroup Theorem for lattices on two-dimensional Euclidean buildings, \emph{arXiv:2605.06163}

\bibitem[Mah]{Mah}
J. Maher, Linear progress in the complex of curves, \emph{Trans. Amer. Math. Soc.} \textbf{362} (2010),  2963--2991.

\bibitem[Mal]{Mal}
A.I. Malcev, On isomorphic matrix representations of infinite groups of matrices (Russian), \emph{Mat. Sb.} {\bf 8} (1940), 405–422, \emph {Amer. Math. Soc. Transl.} {\bf (2) 45} (1965), 1--18.

\bibitem[MT]{MT}
J. Maher, G. Tiozzo, Random walks on weakly hyperbolic groups, \emph{J. Reine Angew. Math.}, \textbf {742} (2018), 187--239.

\bibitem[Osi16]{Osi16}
D. V. Osin, Acylindrically hyperbolic groups, \emph{Trans. Amer. Math. Soc.}, \textbf{368} (2016), pp. 851--888.

\bibitem[Osi17]{Osi17}
D. Osin, Invariant random subgroups of groups acting on hyperbolic spaces. \emph{Proc. Amer. Math. Soc.}, \textbf{145} (2017), no. 8, 3279--88.

\bibitem[Oza]{Oza}
N. Ozawa, Proximality and selflessness for group $C^*$-algebras,  \emph{arXiv:2508.07938v5}.

\bibitem[Pau96]{Paul}
F. Paulin, Un groupe hyperbolique est déterminé par son bord. \emph{J. London Math. Soc.} \textbf{54}(1) (1996), 50--74.

\bibitem[PSZ]{PSZ}
H. Petyt, D. Spriano, A. Zalloum, Hyperbolic models for $\CAT(0)$ spaces, \emph{Adv. in Math.}, \textbf{450} (2024).

\bibitem[Rob]{Rob}
L. Robert, Selfless $C^*$-algebras, \emph{Adv. Math.} \textbf{478} (2025), 110409.

\bibitem[Ser]{Ser}
P. Serre (with H. Bass), Trees, \emph{Springer-Verlag}, 1980.

\bibitem[Sis]{Sis}
A. Sisto, Quasi-convexity of hyperbolically embedded subgroups, \emph{Mathematische Zeitschrift}. \textbf{283} (2016) (3–4), 649--658.

\bibitem[Sun]{Sun}
M. Sunderland, Linear progress with exponential decay in weakly hyperbolic groups, \emph{Groups Geom. Dyn.}, \textbf {14}(2) (2020), 539--566.

\bibitem[TMW]{TMW}
T. Titz Mite, S. Witzel, Non-residually finite $\widetilde{C}_2$-lattices, \emph{	arXiv:2509.05054}.

\bibitem[SWZ]{SWZ}
R. Skipper, S. Witzel, M.C.B. Zaremsky, Simple groups separated by finiteness properties, \emph{Invent. math.} \textbf{215} (2019), 713–740.

\bibitem[Tom]{Tom}
G. Tomanov, Generalized group identities in linear groups,(Russian) \emph{Mat. Sb. (N.S.)}\textbf{1}, (1984) 35--49.

\bibitem[Vid]{Vid}
I. Vigdorovich, Structural properties of reduced $C^*$-algebras associated with higher-rank lattices, \emph{arXiv:2503.12737v3}.

\end{thebibliography}
\end{document}